\documentclass[10pt]{amsart}

\usepackage{amsmath}
\usepackage{amssymb}
\usepackage{verbatim}
\usepackage{enumerate}
\usepackage[all]{xy}
\usepackage{mathrsfs}

\numberwithin{equation}{section}

\theoremstyle{definition}
\newtheorem{defn}[equation]{Definition}

\theoremstyle{remark}
\newtheorem{ex}[equation]{Example}
\newtheorem{remark}[equation]{Remark}

\theoremstyle{theorem}
\newtheorem{thm}[equation]{Theorem}
\newtheorem*{thm*}{Theorem}
\newtheorem{lem}[equation]{Lemma}
\newtheorem{prop}[equation]{Proposition}
\newtheorem{lemma}[equation]{Lemma}
\newtheorem{coro}[equation]{Corollary}

\numberwithin{oracles}{section}

\newtheorem{prob}[equation]{P.}

\newcommand{\exref}[1]{Ex\-am\-ple \ref{#1}}
\newcommand{\thmref}[1]{Theo\-rem \ref{#1}}
\newcommand{\defref}[1]{Def\-i\-ni\-tion \ref{#1}}

\newcommand{\lemref}[1]{Lem\-ma \ref{#1}}
\newcommand{\propref}[1]{Prop\-o\-si\-tion \ref{#1}}
\newcommand{\corref}[1]{Cor\-ol\-lary \ref{#1}}
\newcommand{\figref}[1]{Fig\-ure \ref{#1}}

\newcommand{\remref}[1]{Re\-mark \ref{#1}}
\newcommand{\probref}[1]{P. \ref{#1}}

\providecommand{\normaleq}{\unlhd}
\providecommand{\normal}{\lhd}

\providecommand{\union}{\cup}

\providecommand{\intersect}{\cap}

\DeclareMathOperator{\rad}{rad }
\DeclareMathOperator{\End}{End }

\DeclareMathOperator{\Aut}{Aut}
\DeclareMathOperator{\soc}{soc }
\DeclareMathOperator{\SL}{SL}
\DeclareMathOperator{\PSL}{PSL}
\DeclareMathOperator{\GL}{GL}

\DeclareMathOperator{\im}{im }

\newcommand{\textalgo}[1]{\textsc{#1}}
\newcommand{\Bi}{\mathsf{Bi} }
\newcommand{\Grp}{\mathsf{Grp} }

\newenvironment{block*}
		{\hspace*{0.05\textwidth}\begin{minipage}[t]{0.9\textwidth}}
		{\end{minipage}\\}

\newenvironment{code}[1]
		{\small \begin{center}\begin{minipage}[t]{0.9\textwidth}
			 \texttt{#1}\\ 
			 \texttt{begin}\\ \hspace*{0.05\textwidth}\begin{minipage}[t]{0.9\textwidth}}
		{\end{minipage}\\ \texttt{end.}\\\end{minipage}\end{center} \normalsize}

\newcommand{\cwhile}[1]{\textbf{ while ( #1 )}\\}
\newcommand{\creturn}[1]{\textbf{ return #1 ;}}

\newcommand{\XX}{\mathtt{X}}
\newcommand{\Ops}{\Omega}
\newcommand{\RR}{\mathtt{R}}
\newcommand{\WW}{\mathtt{W}}

\begin{document}

\title{Finding direct product decompositions in polynomial time}
\author{James B. Wilson}
\address{
	Department of Mathematics\\
	The Ohio State University\\
	Columbus, OH 43210
}
\email{wilson@math.ohio-state.edu}
\date{\today}
\thanks{This research was supported in part by NSF Grant DMS 0242983.}

\begin{abstract}
A polynomial-time algorithm is produced which, 
given generators for a group of permutations on a finite set, 
returns a direct product decomposition of the group
into directly indecomposable subgroups.  The process uses bilinear maps and commutative rings to
characterize direct products of $p$-groups of class $2$ and reduces general groups to $p$-groups
using group varieties.  The methods apply to quotients of permutation groups and operator groups as well.
\end{abstract}

\maketitle

\tableofcontents

\section{Introduction}

Forming direct products of groups is an old and elementary way to construct new groups from old ones.  This paper concerns reversing that process by efficiently decomposing a group into a direct product of nontrivial subgroups in a maximal way, i.e. constructing a \emph{Remak decomposition} of the group.  We measure efficiency by describing the time (number of operations) used by an algorithm, as a function of the input size.  Notice that a small set of generating permutations or matrices can specify a group of exponentially larger size; hence, there is some work just to find the order of a group in polynomial time. In the last 40 years, problems of this sort have been attacked with ever increasing dependence on properties of simple groups, and primitive and irreducible actions, cf. \cite{Seress:book}.   A polynomial-time algorithm to construct a Remak decomposition is an obvious addition to those algorithms and, as might be expected, our solution depends on many of those earlier works.  Surprisingly, the main steps involve tools (bilinear maps, commutative rings, and group varieties) that are not standard in Computational Group Theory.

We solve the Remak decomposition problem for permutation groups and describe the method in a framework suitable for other computational settings, such as matrix groups. We prove:
\begin{thm}\label{thm:FindRemak}
There is a deterministic polynomial-time algorithm which, given a permutation group, returns a Remak decomposition of the group.
\end{thm}

It seems natural to solve the Remak decomposition problem by first locating a direct factor of the group, constructing a direct complement, and then recursing on the two factors. Indeed, Luks \cite{Luks:comp} and Wright \cite{Wright:comp} (cf. \thmref{thm:FindComp}) gave polynomial-time algorithms to test if a subgroup is a direct factor and if so to construct a direct complement.  But how do we find a proper nontrivial direct factor to start with?  A critical case for that problem is $p$-groups.  A $p$-group generally has an exponential number of normal subgroups so that searching for direct factors of a $p$-group appears impossible.  

The algorithm for \thmref{thm:FindRemak} does not proceed in the natural fashion just described, and it is more of a construction than a search.  In fact, the algorithm does not produce a single direct factor of the original group until the final step, at which point it has produced an entire Remak decomposition.

It was the study of central products of $p$-groups which inspired the approach we use for \thmref{thm:FindRemak}.  In \cite{Wilson:unique-cent,Wilson:algo-cent}, central products of a $p$-group $P$ of class $2$ were linked, via a bilinear map $\Bi(P)$, to idempotents in a Jordan algebra in a way that explained their size, their $(\Aut P)$-orbits, and demonstrated how to use the polynomial-time algorithms for rings (Ronyai \cite{Ronyai}) to construct fully refined central decompositions all at once (rather than incrementally refining a decomposition). This approach is repeated here, only we replace Jordan algebras with a canonical commutative ring $C(P):=C(\Bi(P))$ (cf. \eqref{eq:Bi} and \defref{def:centroid}).  
Thus, we characterize directly indecomposable $p$-groups of class $2$ as follows:
\begin{thm}\label{thm:indecomp-class2}
If $P$ and $Q$ are finite $p$-groups of nilpotence class $2$ then $C(P\times Q)\cong C(P)\oplus C(Q)$.  Hence, if $C(P)$ is a local ring and $\zeta_1(P)\leq \Phi(P)$, then $P$ is directly indecomposable.  Furthermore, if $P^p=1$ then the converse also holds.
\end{thm}

The algorithm applies the implications of \thmref{thm:indecomp-class2} and begins with the \emph{unique} Remak decomposition of a commutative ring.  This process is repeated across several sections of the group.  Using group varieties we organize the various sections.  Group varieties behave well regarding direct products and come with natural and computable normal subgroups used to create the sections.  To work within these sections of a permutation group we have had to prove \thmref{thm:FindRemak} in the generality of quotients of permutation groups and thus we have used the Kantor-Luks polynomial-time quotient group algorithms \cite{KL:quotient}.  Those methods depend on the Classification of Finite Simple Groups and, in this way, so does \thmref{thm:FindRemak}.  A final generalization of the main result is the need to allow groups with operators $\Ops$ and consider Remak $\Ops$-decompositions.    The most general version of our main result is summarized in \thmref{thm:FindRemak-Q} followed by a variant for matrix groups in \corref{coro:FindRemak-matrix}.

\thmref{thm:FindRemak} was proved in 2008 \cite{Wilson:thesis}.
That same year, with entirely different methods, Kayal-Nezhmetdinov \cite{KN:direct} proved
there is a deterministic polynomial-time algorithm which, given a group $G$ specified by its multiplication table (i.e. the size of input is $|G|^2$),
returns a Remak decomposition of $G$.  The same result follows as a
corollary to \thmref{thm:FindRemak} by means of the regular permutation representation 
of $G$.  \thmref{thm:nearly-linear} states that in that special situation there is a nearly-linear-time algorithm for the task.  

\subsection{Outline}

We organize the paper as follows.  

In Section \ref{sec:background} we introduce the notation and definitions we use throughout.  This includes the relevant group theory background, discussion of group varieties, rings and modules, and a complete listing of the prerequisite tools for \thmref{thm:FindRemak}.

In Section \ref{sec:lift-ext} we show when and how a direct decomposition of a subgroup or quotient group can be extended or lifted to a direct decomposition of the whole group (Sections \ref{sec:induced}--\ref{sec:chains}).  That task centers around the selection of good classes of groups as well as appropriate normal subgroups.  The results in that section are largely non-algorithmic though they lay foundations for the correctness proofs and suggest how the data will be processed by the algorithm for \thmref{thm:FindRemak}.  

Section \ref{sec:lift-ext-algo} applies the results of the earlier section to produce a polynomial-time algorithm which can effect the lifting/extending of direct decompositions of subgroups and quotient groups. First we show how to construct direct $\Omega$-complements of a direct $\Omega$-factor of a group (Section \ref{sec:complements}) by modifying some earlier unpublished work of Luks \cite{Luks:comp} and Wright \cite{Wright:comp}.  Those algorithms answer Problem 2, and (subject to some constraints) also Problem 4 of \cite[p. 13]{KN:direct}.  The rest of the work concerns the algorithm \textalgo{Merge} described in Section \ref{sec:merge} which does the `glueing' together of direct factors from a normal subgroup and its quotient.

In Section \ref{sec:bi} we characterize direct decompositions of $p$-groups of class $2$ by means of an associated commutative ring and prove \thmref{thm:indecomp-class2}.  We close that section with some likely well-known results on groups with trivial centers.

In Section \ref{sec:Remak} we prove \thmref{thm:FindRemak} and its generalization \thmref{thm:FindRemak-Q}.   This is a specific application which demonstrates the general framework setup in Sections \ref{sec:lift-ext} and \ref{sec:lift-ext-algo}.  \thmref{thm:FindRemak-Q} answers Problem 3 of \cite[p. 13]{KN:direct} and 
\corref{coro:FindRemak-matrix} essentially answers Problem 5 of \cite[p. 13]{KN:direct}.

Section \ref{sec:ex} is an example of how the algorithm's main components operate on a specific group.
The execution is explained with an effort to indicate where some of the subtle points in the process arise.

Section \ref{sec:closing} wraps up loose ends and poses some questions.

\section{Background}\label{sec:background}

We begin with a survey of the notation, definitions, and algorithms we use throughout the paper.
Much of the preliminaries can be found in standard texts on Group Theory, consider
\cite[Vol. I \S\S 15--18; Vol. II \S\S 45--47]{Kurosh:groups}.

Typewriter fonts $\XX, \RR$, etc. denote sets without implied properties;
Roman fonts $G$, $H$, etc., denote groups; Calligraphic fonts
$\mathcal{H}, \mathcal{X}$, etc. denote sets and multisets of groups; and the Fraktur
fonts $\mathfrak{X}$, $\mathfrak{N}$, etc. denote classes of groups.

With few exceptions we consider only finite groups.  Functions are evaluated on the right and group actions are denoted exponentially.  We write $\End G$ for the set of endomorphisms of $G$ and $\Aut G$ for the group of 
automorphisms.  The \emph{centralizer} of a subgroup $H\leq G$ is 
$C_G(H)=\{g\in G: H^g=H\}$.  The \emph{upper central series} is 
$\{\zeta_i(G): i\in\mathbb{N}\}$ where $\zeta_0(G)=1$, $\zeta_{i}(G)\normal \zeta_{i+1}(G)$
and $\zeta_{i+1}(G)/\zeta_i(G)=C_{G/\zeta_i(G)}(G/\zeta_i(G))$, for all $i\in\mathbb{N}$.
The commutator of subgroups $H$ and $K$ of $G$ is 
$[H,K]=\langle [h,k]:h\in H, k\in K\rangle$.  The \emph{lower central series} is
$\{\gamma_i(G):i\in\mathbb{Z}^+\}$ where $\gamma_1(G)=G$ and 
$\gamma_{i+1}(G)=[G,\gamma_i(G)]$ for all $i\in\mathbb{Z}^+$.
The \emph{Frattini} subgroup $\Phi(G)$ is the intersection of all maximal subgroups.

\subsection{Operator groups}\label{sec:op-groups}
An $\Ops$-group $G$ is a group, a possibly empty set $\Ops$, and a function 
$\theta:\Ops\to \End G$.  Throughout the paper we write $g^{\omega}$ for 
$g(\omega\theta)$, for all $g\in G$ and all $\omega\in \Ops$.  

With the exception of Section \ref{sec:gen-ops}, we insist that $\Ops\theta\subseteq \Aut G$.

In a natural way, $\Ops$-groups have all the usual definitions of 
$\Ops$-subgroups, quotient $\Ops$-groups, and 
$\Ops$-homomorphisms.  Call $H$ is \emph{fully invariant}, resp. \emph{characteristic} if it is
an $(\End G)-$, resp. $(\Aut G)-$, subgroup. As we insist that $\Ops\theta\subseteq \Aut G$, in this work every characteristic subgroup of $G$ is automatically an $\Ops$-subgroup.
Let $\Aut_{\Ops} G$ denote the $\Ops$-automorphisms of $G$.
We describe normal $\Ops$-subgroups $M$ of $G$ simply as 
$(\Ops\union G)$-subgroup of $G$.

The following characterization is critical to our proofs.
\begin{align}
\label{eq:central}
\Aut_{\Ops\cup G} G & = 
	\{\varphi\in \Aut_{\Ops} G: \forall g\in G, g\varphi \equiv g \pmod{\zeta_1(G)}\}.
\end{align}
It is also evident that $\Aut_{\Ops\cup G} G$ acts as the identity
on $\gamma_2(G)$.
Such automorphisms are called \emph{central}
but for uniformity we described them as  $(\Ops\cup G)$-automorphisms.

We repeatedly use the following property of the $(\Ops\cup G)$-subgroup lattice.
\begin{lemma}[Modular law]\cite[Vol. II \S 44: pp. 91-92]{Kurosh:groups}\label{lem:modular}
If $M$, $H$, and $R$ are $(\Ops\cup G)$-subgroups of an $\Ops$-group $G$ and
$M\leq H$, then $H\cap RM=(H\cap R)M$.
\end{lemma}

\subsection{Decompositions, factors, and refinement}\label{sec:decomps}

Let $G$ be an $\Ops$-group.  
An \emph{$\Ops$-decomposition} of $G$ is a set $\mathcal{H}$ of 
$(\Ops\union G)$-subgroups of $G$ which generates $G$ but no proper 
subset of $\mathcal{H}$ does.  A \emph{direct $\Ops$-decomposition} is 
an $\Ops$-decomposition $\mathcal{H}$ where $H\intersect \langle\mathcal{H}-\{H\}\rangle=1$, 
for all $H\in\mathcal{H}$.    In that case, elements $H$ of $\mathcal{H}$ are direct $\Ops$-factors of $G$ and $\langle\mathcal{H}-\{H\}\rangle$ is a \emph{direct $\Omega$-complement} to $H$.
Call $G$ \emph{directly $\Ops$-indecomposable} 
if $\{G\}$ is the only direct $\Ops$-decomposition of $G$.  Finally, a 
\emph{Remak $\Ops$-decomposition} means a direct $\Ops$-decomposition 
consisting of directly $\Ops$-indecomposable groups. 

Our definitions imply that the trivial subgroup $1$ is not a direct $\Ops$-factor.
Furthermore, the only direct decomposition of $1$ is $\emptyset$ and so $1$ is not 
directly $\Ops$-indecomposable.

We repeatedly use for the following notation.  Fix an $\Ops$-decomposition 
$\mathcal{H}$ of an $\Ops$-group $G$, and an $(\Ops\union G)$-subgroup $M$ of $G$.
Define the sets
\begin{align}
  \mathcal{H}\intersect M & = \{ H\intersect M : H\in\mathcal{H} \} -\{1\},\\
  \mathcal{H}M & = \{ HM : H\in\mathcal{H}\} - \{M\},\textnormal{ and }\\
  \mathcal{H}M/M & = \{ HM/M : H\in\mathcal{H} \} -\{M/M\}.
\end{align}
If $f:G\to H$ is an $\Ops$-homomorphism then define
\begin{align}
	\mathcal{H}f = \{ Hf: H\in\mathcal{H}\}-\{1\}.
\end{align}
Each of these sets consists of $\Ops$-subgroups of $G\intersect M$, $M$, $G/M$, and
$\im f$ respectively.
It is not generally true that these sets are $\Ops$-decompositions.  
In particular, for arbitrary $M$, we should not expect a relationship between the direct
 $\Ops$-decompositions of $G/M$ and those of $G$.   
 
If $\mathfrak{X}$ is a class of groups then set
\begin{align}
	\mathcal{H}\cap \mathfrak{X} & = \{ H\in\mathcal{H}: H\in\mathfrak{X}\},\textnormal{ and}\\
	\mathcal{H}-\mathfrak{X} & = \mathcal{H}-(\mathcal{H}\cap\mathfrak{X}).
\end{align}

An $\Ops$-decomposition $\mathcal{H}$ of $G$ \emph{refines} an $\Ops$-decomposition $\mathcal{K}$ 
of $G$ if for each $H\in\mathcal{H}$, there a unique $K\in\mathcal{K}$ such that
$H\leq K$ and furthermore,
\begin{equation}\label{eq:refine}
	 \forall K\in\mathcal{K},\quad K =\langle H\in\mathcal{H} : H\leq K\rangle.
\end{equation}
When $\mathcal{K}$ is a direct $\Ops$-decomposition, \eqref{eq:refine} implies
the uniqueness preceding the equation.  If $\mathcal{H}$ is a direct 
$\Ops$-decomposition then $\mathcal{K}$ is a direct $\Ops$-decomposition.

An essential tool for us is the so called ``Krull-Schmidt'' theorem for finite groups.
\begin{thm}[``Krull-Schmidt'']\label{thm:KRS}\cite[Vol. II, p. 120]{Kurosh:groups}
If $G$ is an $\Ops$-group and $\mathcal{R}$ and $\mathcal{T}$ are Remak $\Ops$-decompositions of $G$,
then for every $\mathcal{X}\subseteq \mathcal{R}$, there is a $\varphi\in \Aut_{\Ops\cup G} G$ such that
$\mathcal{X}\varphi\subseteq \mathcal{T}$ and $\varphi$ is the identity on $\mathcal{R}-\mathcal{X}$.
In particular, $\mathcal{R}\varphi=\mathcal{X}\varphi\sqcup (\mathcal{R}-\mathcal{X})$ is a Remak
$\Ops$-decomposition of $G$.
\end{thm}

\begin{remark}
The ``Krull-Schmidt'' theorem combines two distinct properties.
First, it is a theorem about exchange (as compared to a basis exchange).
That property was proved by Wedderburn \cite{Wedderburn:direct} in 1909.  
Secondly, it is a theorem about the transitivity of a group action.
That property was the contribution of Remak \cite{Remak:direct} in 1911.
Remak was made aware of Wedderburn's work in the course of publishing his paper
and added to his closing remarks \cite[p. 308]{Remak:direct} that 
Wedderburn's proof contained an unsupported leap (specifically at
\cite[p.175, l.-4]{Wedderburn:direct}).  This leap is not so great
by contemporary standards, for example it occurs in \cite[p.81, l.-12]{Rotman:grp}.
Few references seem to be made to Wedderburn's work following Remak's publication.
  
In 1913, Schmidt \cite{Schmidt:direct} simplified and extended the work of Remak and
in 1925 Krull \cite{Krull:direct} considered direct products of finite and infinite
abelian $\Ops$-groups.  Fitting \cite{Fitting:direct} invented the standard proof
using idempotents, Ore \cite{Ore:lattice1} grounded the concepts in Lattice theory, and 
in several works Kurosh \cite[\S 17,
\S\S 42--47]{Kurosh:groups} and others unified and expanded these results.
By the 1930's direct decompositions of maximum length appear as ``Remak decompositions''
while at the same time the theorem is referenced as ``Krull-Schmidt''.
\end{remark}

\subsection{Free groups, presentations, and constructive presentations}
\label{sec:free}
In various places we use free groups.
Fix a set $\XX\neq \emptyset$ and a group $G$.  Let $G^{\XX}$ denote 
the set of functions from $\XX$ to $G$, equivalently, the set of all $\XX$-tuples of $G$.  

Every $f\in G^{\XX}$ is the restriction of a unique homomorphism $\hat{f}$ from the free group $F(\XX)$ into $G$, that is:
\begin{equation}
	\forall x\in\XX,
		\quad x\hat{f}  = xf.
\end{equation}
We use $\hat{f}$ exclusively in that manner.  As usual we call
$\langle\XX | \RR\rangle$ a \emph{presentation} for a group $G$ with respect 
to $f:\XX\to G$
if $\XX f$ generates $G$ and $\ker \hat{f}$ is the smallest normal subgroup
of $F(\XX)$ containing $\RR$.

Following \cite[Section 3.1]{KLM:Sylow}, 
$\{\langle\XX|\RR\rangle, f:\XX\to G,\ell:G\to F(\XX)\}$ is a \emph{constructive presentation}  for $G$, 
if $\langle\XX | \RR\rangle$ is a presentation for $G$ with
respect to $f$ and $\ell \hat{f}$ is the identity on $G$.
More generally, if $M$ is a normal subgroup of $G$ then  call
$\{\langle\XX| \RR\rangle,f:\XX\to G,\ell:G\to F(\XX)\}$ a \emph{constructive presentation 
for $G$ mod $M$} if $\langle \XX|\RR\rangle$ is a 
presentation of $G/M$ with respect to the induced function 
$\XX\overset{f}{\to} G\to G/M$, also $\ell\hat{f}$ is the identity on $G$, and 
$M\ell\leq \langle \RR^{F(\XX)}\rangle$.

\subsection{Group classes, varieties, and verbal and marginal subgroups}
\label{sec:varieties}
In this section we continue the notation given in Section \ref{sec:free} and
introduce the vocabulary and elementary properties of group varieties studies at length in \cite{HNeumann:variety}.

By a \emph{class of $\Ops$-groups} we shall mean a class which contains
the trivial group and is closed to $\Ops$-isomorphic images.  If $\mathfrak{X}$
is a class of ordinary groups, then $\mathfrak{X}^{\Ops}$ denotes
the subclass of $\Ops$-groups in $\mathfrak{X}$.  

A \emph{variety} $\mathfrak{V}=\mathfrak{V}(\WW)$ is a class of groups defined by a 
set $\WW$ of words, known as \emph{laws}.  Explicitly, $G\in\mathfrak{X}$ if, and only if, every $f\in G^{\XX}$ has $\WW\subseteq \ker \hat{f}$.  We say that $w\in F(\XX)$ is a \emph{consequence}
of the laws $\WW$ if for every $G\in\mathfrak{V}$ and every $f\in G^{\XX}$, 
$w\in \ker \hat{f}$.

The relevance of these classes to direct products is captured in the following:
\begin{thm}[Birkhoff-Kogalovski]\cite[15.53]{HNeumann:variety}\label{thm:BK}
A class of groups is a variety if, and only if, it is nonempty
and is closed to homomorphic images, subgroups,
and direct products (including infinite products).
\end{thm}

Fix a word $w\in F(\XX)$.  We regard $w$ as a function $G^{\XX}\to G$, denoted $w$, where
\begin{equation}\label{eq:w-map}
	\forall f\in G^{\XX},\quad w(f) = w\hat{f}.
\end{equation}
On occasion we write $w(f)$ as $w(g_1,g_2,\dots)$, where $f\in G^{\XX}$ is understood
as the tuple $(g_1,g_2,\dots)$. For example, if $w=[x_1,x_2]$, then $w:G^2\to G$ 
can be defined as $w(g_1,g_2)=[g_1,g_2]$, for all $g_1,g_2\in G$. 

Levi and Hall separately introduced two natural subgroups to associate with the
function $w:G^{\XX}\to G$.
First, to approximate the image of $w$ with a group, we have the \emph{verbal} subgroup 
\begin{equation}\label{eq:def-verbal}
	w(G) = \langle w(f): f\in G^{\XX}\rangle.
\end{equation}
Secondly, to mimic the radical of a multilinear map, we use the \emph{marginal} subgroup 
\begin{equation}\label{eq:marginal}
	w^*(G) = 
	\{ g \in G~:~\forall f'\in \langle g\rangle^{\XX}, \forall f\in G^{\XX},~w(ff')=w(f)\}.
\end{equation}
(To be clear, $ff'\in G^{\XX}$ is the pointwise product: $x(ff')=(xf)(xf')$ for all $x\in \XX$.)
Thus, $w:G^{\XX}\to G$ factors through $w:(G/w^*(G))^{\XX}\to w(G)$.  For a set $\WW$ of words,
the $\WW$-verbal subgroup is $\langle w(G): w\in \WW\rangle$ and the $\WW$-marginal
subgroup is $\bigcap \{w^*(G): w\in \WW\}$.  Observe that for finite sets $\WW$ a single word 
may be used instead, e.g. replace $\WW=\{[x_1,x_2], x_1^2\}
\subseteq F(\{x_1,x_2\})$ with $w=[x_1,x_2]x_3^2\in F(\{x_1,x_2,x_3\})$. 
If we have a variety $\mathfrak{V}$ defined by two sets $\WW$ and ${\tt U}$ of laws, then
every $u\in {\tt U}$ is a consequence of the laws $\WW$.  From the definitions above it 
follows that $u(G)\leq \WW(G)$ and $\WW^*(G)\leq u^*(G)$.  Reversing the roles of
$\WW$ and ${\tt U}$, it follows that $\WW(G)={\tt U}(G)$ and
$\WW^*(G)={\tt U}^*(G)$.  This justifies the notation
\begin{align*}
	\mathfrak{V}(G) & = \mathfrak{V}(\WW)(G)=\WW(G),\\
	\mathfrak{V}^*(G) & = \mathfrak{V}(\WW)^*(G) = \WW^*(G).
\end{align*} 
The verbal and marginal groups are dual in the following sense \cite{Hall:margin}:
for a group $G$, 
\begin{equation}
	 \mathfrak{V}(G)=1\quad \Leftrightarrow\quad G\in\mathfrak{V}
	 	\quad \Leftrightarrow \quad \mathfrak{V}^*(G)=G.
\end{equation}
Also, verbal subgroups are radical, $\mathfrak{V}(G/\mathfrak{V}(G))=1$, and marginal 
subgroups are idempotent, $\mathfrak{V}^*(\mathfrak{V}^*(G))=\mathfrak{V}^*(G)$, but 
verbal subgroups are not generally idempotent and marginal subgroups are not generally radical.

\begin{ex}\label{ex:varieties}
\begin{enumerate}[(i)]
\item  The class $\mathfrak{A}$ of abelian groups is a group variety defined by $[x_1, x_2]$.
The $\mathfrak{A}$-verbal subgroup of a group is the commutator subgroup and the
$\mathfrak{A}$-marginal subgroup is the center.

\item  The class $\mathfrak{N}_c$ of nilpotent groups of class at most $c$ is a group variety 
defined by $[x_1,\dots,x_{c+1}]$ (i.e. $[x_1]=x_1$ and 
$[x_1,\dots,x_{i+1}]=[[x_1,\dots,x_i],x_{i+1}]$, for all $i\in \mathbb{N}$).
Also, $\mathfrak{N}_c(G)=\gamma_{c+1}(G)$ and $\mathfrak{N}_c^*(G)=\zeta_c(G)$ 
\cite[2.3]{Robinson}.  

\item The class $\mathfrak{S}_d$ of solvable groups of derived length at most $d$ is
a group variety defined by $\delta_d(x_1,\dots,x_{2^d})$ where 
$\delta_1(x_1)=x_1$ and for all $i\in\mathbb{N}$,
$$\delta_{i+1}(x_1,\dots,x_{2^{i+1}})
=[\delta_i(x_1,\dots,x_{2^i}),\delta_i(x_{2^i+1},\dots,x_{2^{i+1}})].$$ 
Predictably, $\mathfrak{S}_d(G)=G^{(d)}$ is the 
$d$-th derived group of $G$.  It appears that $\mathfrak{S}_d^*(G)$ is not often used 
and has no name. (This may be good precedent for $\mathfrak{S}_d^*(G)$ can be trivial while $G$ is solvable; thus, the series $\mathfrak{S}^*_1(G)\leq 
\mathfrak{S}^*_2(G)\leq \cdots $ need not be strictly increasing.)
\end{enumerate}
\end{ex}

Verbal and marginal subgroups are characteristic in $G$ and verbal subgroups are also fully 
invariant \cite{Hall:margin}.  So if $G$ is an $\Ops$-group then so is $\mathfrak{V}(G)$.
Moreover, 
\begin{equation}\label{eq:verbal-closure}
	G\in\mathfrak{V}^{\Ops} \textnormal{ if, and only if, $G$ is an $\Ops$-group and }
		\mathfrak{V}(G)=1.
\end{equation}
Unfortunately,  marginal subgroups need not be fully invariant (e.g. the
center of a group).  In their place, we use the $\Ops$-invariant marginal subgroup 
$(\mathfrak{V}^{\Ops})^{*}(G)$, i.e. the largest normal $\Ops$-subgroup of 
$\mathfrak{V}^*(G)$.
Since $\mathfrak{V}$ is closed to subgroups it follows that $(\mathfrak{V}^{\Ops})^{*}(G)\in\mathfrak{V}$.  Furthermore, if $G$ is an $\Ops$-group and $G\in\mathfrak{V}$ then
$\mathfrak{V}^*(G)=G$ and so the $\Ops$-invariant marginal subgroup is $G$.  Thus,
\begin{equation}\label{eq:marginal-closure}
	G\in\mathfrak{V}^{\Ops} \textnormal{ if, and only if, $G$ is an $\Ops$-group and }
		\mathfrak{V}^{*}(G)=G.
\end{equation}
In our special setting all operators
act as automorphisms and so the invariant marginal subgroup is indeed the marginal subgroup.
Nevertheless, to avoid confusion insist that the marginal subgroup of a variety 
of $\Ops$-groups refers to the $\Ops$-invariant marginal subgroup.  

\subsection{Rings, frames, and modules}\label{sec:rings}
We involve some standard theorems for associative unital finite rings and modules.  
Standard references for our uses include \cite[Chapters 1--3]{Herstein:rings} and 
\cite[Chapters I--II, V.3]{Jacobson:Lie}.  Throughout this section $R$ denotes
a finite associative unital ring.

A $e\in R-\{0\}$ is \emph{idempotent} if $e^2=e$.  An idempotent is
\emph{proper} if it is not $1$ (as we have excluded $0$ as an idempotent).  Two idempotents $e,f\in R$
are \emph{orthogonal} if $ef=0=fe$.  An idempotent
is \emph{primitive} if it is not the sum of two orthogonal idempotents.  Finally, 
a \emph{frame} $\mathcal{E}\subseteq R$ is a set of pairwise orthogonal primitive idempotents
of $R$ which sum to $1$.  We use the following properties.
\begin{lem}[Lifting idempotents]\label{lem:lift-idemp}
Let $R$ be a finite ring.
\begin{enumerate}[(i)]
\item If $e\in R$ such that $e^2-e\in J(R)$ (the Jacobson radical) 
then for some $n\leq \log_2 |J(R)|$, $(e^2-e)^n=0$ and
\begin{equation*}
	\hat{e}= \sum_{i=0}^{n-1}\binom{2n-1}{i} e^{2n-1-i}(1-e)^i
\end{equation*}
is an idempotent in $R$.  Furthermore, $\widehat{1-e}=1-\hat{e}$.
\item $\mathcal{E}$ is a frame of $R/J(R)$ then $\hat{\mathcal{E}}=\{\hat{e}:e\in\mathcal{E}\}$
is a frame of $R$.
\item Frames in $R$ are conjugate by a unit in $R$; in particular, if $R$ is commutative then 
$R$ has a unique frame.
\end{enumerate}
\end{lem}
\begin{proof}
Part (i) is verified directly, compare \cite[(6.7)]{Curtis-Reiner}.
Part (ii) follows from induction on (i).  For (iii) see \cite[p. 141]{Curtis-Reiner}.
\end{proof}
  
If $M$ is an $R$-module and $e$ is an idempotent of $\End_R M$ then $M=Me\oplus M(1-e)$.
Furthermore, if $M=E\oplus F$ as an $R$-module, then the projection $e_E:M\to M$ with kernel $F$ and image $E$ is an idempotent endomorphism of $M$.  Thus, every direct $R$-decomposition $\mathcal{M}$ of $M$ is parameterized by a set $\mathcal{E}(\mathcal{M})=\{e_E : E\in\mathcal{M}\}$ of pairwise orthogonal idempotents of $\End_R M$ which sum to $1$.  Remak $R$-decompositions of $M$ correspond to frames of $\End_R M$.

\subsection{Polynomial-time toolkit}
\label{sec:tools}

We use this section to specify how we intend to compute with groups of permutations.
We operate in the context of quotients of permutation groups and borrow from the large
library of polynomial-time algorithms for this class of groups.  We detail
the problems we use in our proof of \thmref{thm:FindRemak} so that in principle any computational
domain with polynomial-time algorithms for these problems will admit a theorem 
similar to \thmref{thm:FindRemak}.  The majority of algorithms which we cite do not provide specific
estimates on the polynomial timing.  Therefore, our own main theorems will not have specific estimates.

The group $S_n$ denotes the permutations on $\{1,\dots,n\}$.  Given $\XX\subseteq S_n$, a
\emph{straight-line program} over $\XX$ is a recursively defined
function on $\XX$ which evaluates to a word over $\XX$, but can be stored and evaluated in an
efficient manner; see \cite[p. 10]{Seress:book}.  To simplify notation we treat these as elements in 
$S_n$.   

Write $\mathbb{G}_n$ for the class of  groups $G$ 
encoded by $(\XX:\RR)$ where $\XX\subseteq S_n$
and $\RR$ is a set of straight-line programs such that
\begin{equation}\label{eq:def-G}
	G=\langle\XX\rangle/N,\qquad N:=\left\langle \RR^{\langle \XX\rangle}\right\rangle\leq \langle\XX\rangle\leq S_n.
\end{equation}
The notation $\mathbb{G}_n$ intentionally avoids reference to the permutation domain as the algorithms we consider can be adapted to other computational domains. Also, observe that a group $G\in\mathbb{G}_n$
may have no small degree permutation representation.  For example, the extraspecial group $2^{1+2n}_+$ is
a quotient of $D_8^{n}\leq S_{4n}$; yet, the smallest faithful permutation representation 
of $2^{1+2n}_+$  has degree $2^n$ \cite[Introduction]{Neumann:perm-grp}.
It is misleading to think of $\RR$ in \eqref{eq:def-G} as relations for the generators $\XX$; 
indeed, 
elements in $\XX$ are also permutations and so there are relations implied on $\XX$
which may not be implied by $\RR$.
We write $\ell(\RR)$ for the sum of the lengths of straight-line programs in $\RR$. 
 
A homomorphism $f:G\to H$ of groups 
$G=(\XX :\RR),H=({\tt Y}:{\tt S}) \in\mathbb{G}_n$
is encoded by storing $\XX f$ as straight-line programs in ${\tt Y}$.
An $\Omega$-group $G$ is encoded by $G=(\XX:\RR)\in\mathbb{G}_n$ along with a function
$\theta:\Omega\to \End G$.  We write $\mathbb{G}_n^{\Omega}$
for the set of $\Omega$-groups encoded in that fashion.

A \emph{polynomial-time} algorithm with input 
$G=(\XX:\RR)\in \mathbb{G}_n^{\Omega}$ returns an output using a polynomial 
in $|\XX|n+\ell(\RR)+\ell(\Omega)$ number of steps.  In some cases
$|\XX|n+\ell(\RR)\in O(\log |G|)$; so, $|G|$ can be exponentially larger than the input size.
When we say ``given an $\Omega$-group $G$'' we shall mean $G\in\mathbb{G}_n^{\Omega}$.

Our objective in this paper is to solve the following problem.

\begin{prob}{\sc Remak-$\Ops$-Decomposition}\label{prob:FindRemak}
\begin{description}
\item[Given] an $\Omega$-group $G$,
\item[Return] a Remak $\Ops$-decomposition for $G$.
\end{description}
\end{prob}

The problems \probref{prob:Order}--\probref{prob:MinSNorm} have polynomial-time solutions 
for groups in $\mathbb{G}_n^{\Omega}$.

\begin{prob}{\sc Order}\label{prob:Order}\cite[P1]{KL:quotient}
\begin{description}
\item[Given] a group $G$, 
\item[Return] $|G|$.
\end{description}
\end{prob}

\begin{prob}{\sc Member}\label{prob:Member}\cite[3.1]{KL:quotient}
\begin{description}
\item[Given] a group $G$, a subgroup $H=(\XX':\RR')$ of $G$, and $g\in G$,
\item[Return] false if $g\notin H$; else, a straight-line program in $\XX'$ reaching $g\in H$.
\end{description}
\end{prob}

We require the means to solve systems of linear equations, or determine that no solution exists,
in the following generalized setting.
\begin{prob}{\sc Solve}\label{prob:Solve}\cite[Proposition 3.7]{KLM:Sylow}
\begin{description}
\item[Given] a group $G$, an abelian normal subgroup $M$, a function
$f\in G^{\XX}$ of constants in $G$, and a set $\WW\subseteq F(\XX)$ of words encoded via straight-line programs;
\item[Return] false if $w(f\mu)\neq 1$ for all $\mu\in M^{\XX}$; else, generators
for the solution space $\{\mu\in M^{\XX} : w(f\mu)=1\}$.
\end{description}
\end{prob}

\begin{prob}{\sc Presentation}\label{prob:pres}\cite[P2]{KL:quotient}
\begin{description}
\item[Given] given a group $G$ and a normal subgroup $M$,
\item[Return] a constructive presentation $\{\langle\XX|\RR\rangle, f,\ell\}$ for $G$ mod $M$.
\end{description}
\end{prob}

\begin{prob}{\sc Minimal-Normal}\label{prob:MinNorm}\cite[P11]{KL:quotient}
\begin{description}
\item[Given] a group $G$,
\item[Return] a minimal normal subgroup of $G$.
\end{description}
\end{prob}

\begin{prob}{\sc Normal-Centralizer}\label{prob:CentNorm}\cite[P6]{KL:quotient}
\begin{description}
\item[Given] a group $G$ and a normal subgroup $H$,
\item[Return] $C_G(H)$.
\end{description}
\end{prob}

\begin{prob}{\sc Primary-Decomposition}\label{prob:Primary}
\begin{description}
\item[Given] an abelian group $A\in \mathbb{G}_n$,
\item[Return] a primary decomposition for $A=\bigoplus_{v\in{\tt B}} \mathbb{Z}_{p^e} v$,
where for each $v\in {\tt B}$, $|v|=p^e$ for some prime $p=p(v)$.
\end{description}
\end{prob}

We call $\mathcal{X}$, as in {\sc Primary-Decomposition}, a \emph{basis} for $A$.
The polynomial-time solution of {\sc Primary-Decomposition} is routine.  
Let $A=(\XX : \RR)\in \mathbb{G}_n$. Use {\sc Order} to compute $|A|$.
As $A$ is a quotient of a permutation group, the primes dividing $|A|$ are less than $n$.
Thus, pick a prime $p\mid |A|$ and write $|A|=p^e m$ where $(p,m)=1$.  Set
$A_p=A^{m}$. 
Using {\sc Member} build a basis ${\tt B}_p$ for $A_p$ by unimodular linear algebra.  
(Compare \cite[Section 2.3]{Wilson:algo-cent}.)  The return is $\bigsqcup_{p\mid |A|} {\tt B}_p$.

We involve some problems for associative rings.  For ease we assume that all 
rings $R$ are finite of characteristic $p^e$ and specified with a basis 
${\tt B}$ over $\mathbb{Z}_{p^e}$.
To encode the multiplication in $R$ we store structure constants
$\{\lambda_{xy}^z \in \mathbb{Z}_{p^e} : x,y,z\in {\tt B}\}$ which are defined so that:
\begin{equation*}
	\left(\sum_{x\in\mathcal{X}} r_x x\right)\left(\sum_{y\in\mathcal{X}} s_y y\right)
		=\sum_{z\in{\tt B}} \left(\sum_{x,y\in\mathcal{X}} r_x \lambda_{xy}^{z} s_y\right)z
\end{equation*}
where, for all $x$ and all $y$ in ${\tt B}$, $r_x,s_y\in\mathbb{Z}_{p^e}$.

\begin{prob}{\sc Frame}\label{prob:Frame}
\begin{description}
\item[Given] an associative unital ring $R$,
\item[Return] a frame of $R$.
\end{description}
\end{prob}
{\sc Frame} has various nondeterministic solutions \cite{EG:fast-alge,Ivanyos:fast-alge} with
astonishing speed.  However, we need a deterministic solution such as in the work of Ronyai.

\begin{thm}[Ronyai \cite{Ronyai}]\label{thm:Frame}
For rings $R$ specified as an additive group in $\mathbb{G}_n$ with a basis 
and with structure constants with respect to the basis,
{\sc Frame} is solvable in polynomial-time in $p+n$ where $|R|=p^n$.
\end{thm}
\begin{proof}
First pass to ${\bf R}=R/pR$ and so create an algebra
over the field $\mathbb{Z}_p$.  Now \cite[Theorem 2.7]{Ronyai} gives a 
deterministic polynomial-time algorithm which finds a basis for the Jacobson
radical of ${\bf R}$.  This allows us to pass to ${\bf S}={\bf R}/J({\bf R})$,
which is isomorphic to a direct product of matrix rings over finite fields.  
Finding the frame for ${\bf S}$ can be done by finding the minimal ideals 
$\mathcal{M}$ of ${\bf S}$ \cite[Corollary 3.2]{Ronyai}.  Next, for each
$M\in\mathcal{M}$, build an isomorphism $M\to M_n(\mathbb{F}_q)$ \cite[Corollary 5.3]{Ronyai}
and choose a frame of idempotents from $M_n(\mathbb{F}_q)$ and let $\mathcal{E}_M$
be the pullback to $M$.  Set $\mathcal{E} =\bigsqcup_{M\in\mathcal{M}} \mathcal{E}_M$
noting that $\mathcal{E}$ is a frame for ${\bf} S$.  Hence, use the power series of
\lemref{lem:lift-idemp} to lift the frame $\mathcal{E}$ to a frame $\hat{\mathcal{E}}$ for
$R$.
\end{proof}

With \thmref{thm:Frame} we setup and solve a special instance of \thmref{thm:FindRemak}.

\begin{prob}{\sc Abelian.Remak-$\Ops$-Decomposition}\label{prob:FindRemak-abelian}
\begin{description}
\item[Given] an abelian $\Omega$-group $A$,
\item[Return] a Remak $\Ops$-decomposition for $A$.
\end{description}
\end{prob}

\begin{coro}\label{coro:FindRemak-abelian}
{\sc Abelian.Remak-$\Ops$-Decomposition} has a polynomial-time solution.
\end{coro}
\begin{proof}
Let $A\in\mathbb{G}_n^{\Ops}$ be abelian.

\emph{Algorithm.}
Use {\sc Primary-Decomposition} to write $A$ in a primary decomposition.
For each prime $p$ dividing $|A|$, let $A_p$ be the $p$-primary component.
Write a basis for $\End A_p$ (noting that $\End A_p$ is a checkered
matrix ring determined completely by the Remak decomposition of $A_p$ as a $\mathbb{Z}$-module 
\cite[p. 196]{McDonald:fin-ring})
and use {\sc Solve} to find a basis for $\End_{\Omega} A$.  Finally, use 
{\sc Frame} to find
a frame $\mathcal{E}_p$ for $\End_{\Ops} A_p$.  Set $\mathcal{A}_p=\{Ae: e\in\mathcal{E}\}$.
Return $\bigsqcup_{p\mid |A|} \mathcal{A}_p$.

\emph{Correctness.} Every direct $\Ops$-decomposition of $A$ corresponds to a set
of pairwise orthogonal idempotents in $\End_{\Ops} A$ which sum to $1$.  Furthermore,
Remak $\Ops$-decomposition correspond to frames.

\emph{Timing.}
The polynomial-timing follows from \thmref{thm:Frame} together with the
observation that $p\leq n$ whenever $A\in \mathbb{G}_n$.
\end{proof}

\begin{remark}\label{rem:matrix}
In the context of groups of matrices our solution to 
{\sc Abelian.Remak-$\Ops$-decomposition} is impossible as it invokes
integer factorization and {\sc Member} is a version of a discrete log problem in that
case.  The primes involved in the orders
of matrix groups can be exponential in the input length and so these two
routines are infeasible.  For solvable matrix groups whose primes are bound and so called
$\Gamma_d$-matrix groups the required problems
in this section have polynomial-time solutions, cf. \cite{Luks:mat,Taku}.
\end{remark}

\begin{prob}{\sc Irreducible}\label{prob:Irreducible}\cite[Corollary 5.4]{Ronyai}
\begin{description}
\item[Given] an associative unital ring $R$, an abelian group $V$, and a homomorphism $\varphi:R\to \End V$,
\item[Return] an irreducible $R$-submodule of $V$.
\end{description}
\end{prob}
As with the algorithm {\sc Frame}, there are nearly optimal nondeterministic methods for {\sc Irreducible},
for example, the MeatAxe \cite{Meataxe1,Meataxe2}; however, we are concerned here with a deterministic
method solely.

\begin{prob}{\sc Minimal-$\Ops$-Normal}\label{prob:MinSNorm}
\begin{description}
\item[Given] an $\Omega$-group $G$ where $\Omega$ acts on $G$ as automorphisms,
\item[Return] a minimal $(\Ops\cup G)$-subgroup of $G$.
\end{description}
\end{prob}
\begin{prop}
{\sc Minimal-$\Ops$-Normal} has a polynomial-time solution.
\end{prop}
\begin{proof}
Let $G=(\XX: \RR)\in\mathbb{G}_n^{\Ops}$.

\emph{Algorithm.}
Use {\sc Minimal-Normal} to compute a minimal normal subgroup $N$ of $G$.
Using {\sc Member}, run the following transitive closure: set $M:=N$, then 
while there exists $w\in \Ops\cup \XX$ such that $M^w\neq M$, set $N=\langle M,M^w\rangle$.
Now $M=\langle N^{\Ops\cup G}\rangle$.  If $N$ is non-abelian then return $M$; otherwise,
treat $M$ as an $(\Ops\cup G)$-module and use {\sc Irreducible} to find an irreducible 
$(\Ops\cup G)$-submodule $K$ of $M$.  Return $K$.

\emph{Correctness.}
Note that $M=\langle N^{\Ops\cup G}\rangle=N N^{w_1} N^{w_2}\cdots N^{w_t}$ for some 
$w_1,\dots,w_t\in \langle \Ops\theta\rangle\ltimes G\leq \Aut G\ltimes G$.  As $N$ is minimal normal,
so is each $N^{w_i}$ and therefore $M$ is a direct product of isomorphic simple groups.
If $N$ is non-abelian then the normal subgroups of $M$ are its direct factors and furthermore, every
direct factor $F$ of $M$ satisfies $M=\langle F^{\Ops\cup G}\rangle$.  If $N$ is abelian then $N\cong\mathbb{Z}_p^d$
for some prime $p$.  A minimal $(\Ops\cup G)$-subgroup of $N$ is therefore an irreducible 
$(\Ops\cup G)$-submodule of $V$.

\emph{Timing.}  First the algorithm executes a normal closure using the polynomial-time algorithm
{\sc Member}.  We test if $N$ is abelian by computing the commutators of the generators.  The final
step is the polynomial-time algorithm {\sc Irreducible}.
\end{proof}

\section{Lifting, extending, and matching direct decompositions}\label{sec:lift-ext}

We dedicate this section to understanding when a direct decomposition of a quotient or subgroup
lifts or extends to a direct decomposition of the whole group.  Ultimately we plan these
ideas for use in the algorithm for \thmref{thm:FindRemak}, but the questions have taken on independent intrigue.  The highlights of this section are Theorems \ref{thm:Lift-Extend} and \ref{thm:chain} and Corollaries \ref{coro:canonical-graders} and \ref{coro:canonical-grader-II}.

Fix a short exact sequence of $\Ops$-groups:
\begin{equation}\label{eq:SES}
\xymatrix{
1\ar[r] & K \ar[r]^{i} & G\ar[r]^{q} & Q\ar[r] & 1.
}
\end{equation}
With respect to \eqref{eq:SES} we study instances of the following problems.
\begin{description}
\item[Extend] for which  direct $(\Ops\cup G)$-decomposition $\mathcal{K}$ of $K$,
is there a Remak $\Ops$-decomposition $\mathcal{R}$ of $G$ such that
$\mathcal{K}i = \mathcal{R}\cap (Ki)$.

\item[Lift] for which direct $(\Ops\cup G)$-decomposition $\mathcal{Q}$ of
$Q$, is there a Remak $\Ops$-decomposition $\mathcal{R}$ of $G$ such that
$\mathcal{Q} = \mathcal{R}q$.

\item[Match] for which pairs $(\mathcal{K},\mathcal{Q})$ of direct $(\Ops\cup G)$-decompositions of
$K$ and $Q$ respectively, is there a Remak $\Ops$-decomposition of $G$ which is an extension of 
$\mathcal{K}$ and a lift of $\mathcal{Q}$, i.e. $\mathcal{K}i=\mathcal{R}\cap (Ki)$ and 
$\mathcal{Q}=\mathcal{R}q$.
\end{description}

Finding direct decompositions which extend or lift is surprisingly easy
(\thmref{thm:Lift-Extend}), but we have had only narrow success in finding matches.  
Crucial exceptions are $p$-groups of class $2$ (\thmref{thm:Match-class2})
where the problem reduces to commutative ring theory.  

\subsection{Graded extensions}\label{sec:induced}
In this section we place some reasonable parameters on the short exact sequences
which we consider in the role of \eqref{eq:SES}.  
This section depends mostly on the material of Sections 
\ref{sec:op-groups}--\ref{sec:decomps}.

\begin{lemma}\label{lem:induced}
Let $G$ be a group with a direct $\Ops$-decomposition $\mathcal{H}$.  If $X$ is an 
$(\Ops\cup G)$-subgroup of $G$ and $X=\langle \mathcal{H}\intersect X\rangle$, then
\begin{enumerate}[(i)]
\item $\mathcal{H}\intersect X$ is a direct $\Ops$-decomposition of $X$,
\item
$\mathcal{H}X/X$ is a direct $\Ops$-decomposition of $G/X$, 
\item
$\mathcal{H}-\{H\in\mathcal{H} : H\leq X\}$, $\mathcal{H}X$, and $\mathcal{H}X/X$
are in a natural bijection, and
\item if $Y$ is an $(\Ops\cup G)$-subgroup of $G$ with $Y=\langle\mathcal{H}\cap Y\rangle$
then $\mathcal{H}\cap (X\cap Y)=\langle\mathcal{H}\cap (X\cap Y)\rangle$ and
$\mathcal{H}\cap XY=\langle \mathcal{H}\cap XY\rangle$.
\end{enumerate}
\end{lemma}
\begin{proof} For (i), $(H\cap X)\cap \langle \mathcal{H}\cap X-
\{H\cap X\}\rangle=1$ for all $H\cap X\in\mathcal{H}\cap X$.  
For (ii), let $|\mathcal{H}|>1$, take $H\in\mathcal{H}$, and set 
$J=\langle\mathcal{H}-\{H\}\rangle$.  From (i):
$HX\intersect JX=(H\times (J\intersect X))\intersect 
((H\intersect X)\times J)=(H\intersect X)\times (J\intersect X)=X$. 
For (iii), the functions $H\mapsto HX\mapsto HX/X$, for each 
$H\in\mathcal{H}-\{H\in\mathcal{H}: H\leq X\}$, suffice.  Finally for (iv), 
let $g\in X\cap N$.  So there are unique $h\in H$ and $k\in\langle\mathcal{H}-\{H\}\rangle$
with $g=hk$.  By (i) and the uniqueness, we get that $h\in (H\cap X)\cap (H\cap Y)$ and 
$k\in \langle \mathcal{H}-\{H\}\rangle \cap (X\cap Y)$.  So 
$g\in \langle \{H\cap (X\cap N),\langle \mathcal{H}-\{H\}\rangle \cap (X\cap Y)\}\rangle$.
By induction on $|\mathcal{H}|$, $X\cap Y\leq \langle \mathcal{H}\cap (X\cap Y)\rangle
\leq X\cap N$.  The last argument is similar.
\end{proof}

We now specify which short exact sequence we consider.
\begin{defn}\label{def:graded}
A short exact sequence $1\to K\overset{i}{\to} G\overset{q}{\to} Q\to 1$ of $\Ops$-groups
is \emph{$\Ops$-graded} if for all (finite) direct $\Omega$-decomposition 
$\mathcal{H}$ of $G$, it follows that $Ki  = \langle \mathcal{H}\cap (Ki) \rangle$.
Also, if $M$ is an $(\Ops\cup G)$-subgroup of $G$ such that 
the canonical short exact sequence $1\to M\to G\to G/M\to 1$ is $\Ops$-graded
then we say that $M$ is $\Ops$-graded.
\end{defn}
\lemref{lem:induced} parts (i) and (ii) imply that every direct $\Ops$-decomposition of
$G$ induces direct $\Ops$-decompositions of $K$ and $Q$ whenever 
$1\to K\overset{i}{\to} G\overset{q}{\to} Q\to 1$ is $\Ops$-graded.  The universal quantifier
in the definition of graded exact sequences may seem difficult to satisfy; nevertheless, in 
Section \ref{sec:direct-ext} we show many well-known subgroups are graded, for example the commutator
subgroup.

\begin{prop}\label{prop:graded-lat}
\begin{enumerate}[(i)]
\item If $M$ is an $\Ops$-graded subgroup of $G$ and $N$ an $(\Ops\cup G)$-graded subgroup
of $M$, then $N$ is an $\Ops$-graded subgroup of $G$.
\item
The set of $\Ops$-graded subgroups of $G$ is a modular
sublattice of the lattice of $(\Ops\cup G)$-subgroups of $G$.
\end{enumerate}
\end{prop}
\begin{proof}
For (i), if $\mathcal{H}$ is a direct $\Ops$-decomposition of $G$ then by 
\lemref{lem:induced}(i), $\mathcal{H}\cap M$ is direct $\Ops$-decomposition of $M$
and so $\mathcal{H}\cap N=(\mathcal{H}\cap M)\cap N$ is a direct
$\Ops$-decomposition of $N$.  Also (ii) follows from \lemref{lem:induced}(iv).
\end{proof}

\begin{lem}\label{lem:KRS}
For all Remak $\Ops$-decomposition $\mathcal{H}$ and all 
direct $\Ops$-decomposition $\mathcal{K}$ of $G$,
\begin{enumerate}[(i)]
\item $\mathcal{H}M$ refines $\mathcal{K}M$ for all 
$(\Ops\cup G)$-subgroups $M\geq \zeta_1(G)$,
\item $\mathcal{H}\intersect M$ refines $\mathcal{K}\intersect M$ for all 
$(\Ops\cup G)$-subgroups $M \leq \gamma_2(G)$.
\end{enumerate}
\end{lem}
\begin{proof} Let $\mathcal{T}$ be a Remak $\Ops$-decomposition of $G$ which refines
$\mathcal{H}$.  By \thmref{thm:KRS}, there is a $\varphi\in \Aut_{\Ops\cup G} G$ such that
$\mathcal{R}\varphi=\mathcal{T}$.  Form \eqref{eq:central} it follows that
$\mathcal{R}\zeta_1(G)=\mathcal{R}\zeta_1(G)\varphi=\mathcal{T}\zeta_1(G)$ and
$\mathcal{R}\cap \gamma_2(G)=(\mathcal{R}\cap\gamma_2(G))\varphi=\mathcal{T}\cap\gamma_2(G)$.
\end{proof}

\begin{thm}\label{thm:Lift-Extend}
Given the commutative diagram in \figref{fig:LIFT-EXT} which is exact and $\Ops$-graded 
in all rows and all columns, the following hold.
\begin{figure}
\begin{equation*}
\xymatrix{
		& 				  &   1	     &	1 & \\
1\ar[r] & K\ar[r]^{i} & G \ar[r]^q\ar[u] & Q\ar[r]\ar[u] & 1\\
1\ar[r] & \hat{K}\ar[r]^{\hat{i}}\ar[u]^{j} & G \ar[r]^{\hat{q}} \ar@{=}[u] 
	& \hat{Q}\ar[r]\ar[u]^{r} & 1\\
		& 	1\ar[u]			  &   1\ar[u]	     &	 & \\
}
\end{equation*}
\caption{A commutative diagram of $\Ops$-groups which is exact and $\Ops$-graded in all rows and all columns.}
\label{fig:LIFT-EXT}
\end{figure}
\begin{enumerate}[(i)]
\item
If $\zeta_1(\hat{Q})r=1$ then for every Remak $\Ops$-decomposition $\hat{\mathcal{Q}}$ of $\hat{Q}$ and 
every Remak $\Ops$-decomposition $\mathcal{H}$ of $G$, $\mathcal{Q}:=\hat{\mathcal{Q}}r$  refines 
$\mathcal{H}q$.
In particular, $\mathcal{H}$ lifts a partition of $\mathcal{Q}$ which is unique to $(G, i,q)$.

\item
If $\gamma_2(K)\leq \hat{K}j$ then for every Remak $(\Ops\cup G)$-decomposition $\mathcal{K}$ of $K$ and
every Remak  $\Ops$-decomposition $\mathcal{H}$ of $G$, $\mathcal{K}i\cap \hat{K}\hat{i}$
refines $\mathcal{H}\cap \left(\hat{K}\hat{i}\right)$.
In particular, $\mathcal{H}$ extends a partition of $\hat{\mathcal{K}}:=
(\mathcal{K}\cap \hat{K}j)j^{-1}$ which is unique to $(G,\hat{i},\hat{q})$.
\end{enumerate}
\end{thm}
\begin{proof}
Fix a Remak $\Ops$-decomposition $\mathcal{H}$ of $G$.

As $\hat{K}$ and $K$ are $\Ops$-graded, it follows that $\mathcal{H}\hat{q}$ is 
a direct $\Ops$-decompositions of $\hat{Q}$ (\lemref{lem:induced}(ii)).  
Let $\mathcal{T}$ be a Remak $\Ops$-decomposition of $\hat{Q}$ which refines
$\mathcal{H}\hat{q}$.  By \lemref{lem:KRS}(i), $\hat{\mathcal{Q}}\zeta_1(\hat{Q})
=\mathcal{T}\zeta_1(\hat{Q})$ and so $\hat{\mathcal{Q}}r = \mathcal{T}r$.  Therefore, 
$\mathcal{Q}:=\hat{\mathcal{Q}}r$ refines $\mathcal{H}\hat{q}r=\mathcal{H}q$.  That proves (i).

To prove (ii), by \lemref{lem:induced}(i) we have that  
$\mathcal{H}\cap (Ki)$ is a direct $(\Ops\cup G)$-decompositions of $Ki$. 
Let $\mathcal{T}$ be a Remak $(\Ops\cup G)$-decomposition
of $Ki$ which refines $\mathcal{H}\cap (Ki)$.  By \lemref{lem:KRS}(ii), 
$\hat{\mathcal{K}}=\mathcal{K}i\cap (\hat{K}\hat{i})=\mathcal{T}\cap \left(\hat{K}\hat{i}\right)$.
Therefore, $\mathcal{K}i \cap \left(\hat{K}\hat{i}\right)$ refines 
$\mathcal{H}\cap \left(\hat{K}\hat{i}\right)$.
\end{proof}

\thmref{thm:Lift-Extend} implies the following special setting where the match problem can be answered.
This is the only instance we know where the matching problem can be solved without considering
the cohomology of the extension.

\begin{coro}\label{coro:Match-perfect-centerless}
If $1\to K\to G\to Q\to 1$ is a $\Ops$-graded short exact sequence where
$K=\gamma_2(K)$ and $\zeta_1(Q)=1$; then for every 
Remak $(\Ops\cup G)$-decomposition $\mathcal{K}$ of $K$ and $\mathcal{Q}$ of $Q$,
there are partitions $[\mathcal{K}]$ and $[\mathcal{Q}]$ unique to the short exact sequence
such that every Remak $\Ops$-decomposition $\mathcal{H}$ of $G$ matches 
$([\mathcal{K}],[\mathcal{Q}])$.
\end{coro}

\subsection{Direct classes, and separated and refined decompositions}\label{sec:direct-class}
In this section we begin our work to consider the extension, lifting, and matching problems
in a constructive fashion.  We introduce classes of groups which are
closed to direct products and direct decompositions and show how to use these classes to
control the exchange of direct factors.

\begin{defn}
A class $\mathfrak{X}$ (or $\mathfrak{X}^{\Ops}$ if context demands) 
of $\Ops$-groups is \emph{direct} if $1\in\mathfrak{X}$,
and $\mathfrak{X}$ is closed to $\Ops$-isomorphisms, as well as the following:
\begin{enumerate}[(i)]
\item if $G\in\mathfrak{X}$ and $H$ is a direct $\Ops$-factor of $G$, then $H\in\mathfrak{X}$,
and 
\item if $H,K\in\mathfrak{X}$ then $H\times K\in\mathfrak{X}$.
\end{enumerate}
\end{defn}

Every variety of $\Ops$-groups is a direct class by \thmref{thm:BK} and to specify the finite
groups in a direct class it is sufficient to specify the directly $\Ops$-indecomposable group
it contains.  However, in practical terms there are few settings where the directly
$\Ops$-indecomposable groups are known.

\begin{defn}
A direct $\Ops$-decomposition $\mathcal{H}$ is \emph{$\mathfrak{X}$-separated} if
for each $H\in\mathcal{H}-\mathfrak{X}$, if $H$ has a direct $\Ops$-factor
$K$, then $K\notin\mathfrak{X}$.  If additionally every member of 
$\mathcal{H}\cap \mathfrak{X}$ is directly $\Ops$-indecomposable, then
$\mathcal{H}$ is \emph{$\mathfrak{X}$-refined}.
\end{defn}

\begin{prop}\label{prop:direct-class}
Suppose that $\mathfrak{X}$ is a direct class of $\Ops$-groups, $G$ an $\Ops$-group,
and $\mathcal{H}$ a direct $\Ops$-decomposition of $G$.  The following hold.
\begin{enumerate}[(i)]
\item $\langle\mathcal{H}\cap\mathfrak{X}\rangle\in\mathfrak{X}$.

\item
If $\mathcal{H}$ is $\mathfrak{X}$-separated and $\mathcal{K}$ is a direct $\Ops$-decomposition of $G$ which refines $\mathcal{H}$, then $\mathcal{K}$ is $\mathfrak{X}$-separated.

\item $\mathcal{H}$ is a $\mathfrak{X}$-separated
if, and only if, $\{\langle\mathcal{H}-\mathfrak{X}\rangle,
\langle\mathcal{H}\cap\mathfrak{X}\rangle\}$ is $\mathfrak{X}$-separated.

\item Every Remak $\Ops$-decomposition is $\mathfrak{X}$-refined.

\item If $\mathcal{H}$ and $\mathcal{K}$ are $\mathfrak{X}$-separated direct
$\Ops$-decompositions of $G$ then 
$(\mathcal{H}-\mathfrak{X})\sqcup (\mathcal{K}\cap\mathfrak{X})$ is an 
$\mathfrak{X}$-separated direct $\Ops$-decomposition of $G$.
\end{enumerate}
\end{prop}
\begin{proof}
First, (i) follows as $\mathfrak{X}$ is closed to direct $\Ops$-products. 

For (ii), notice that a direct $\Ops$-factor of a $K\in \mathcal{K}$ is also a 
direct $\Ops$-factor of the unique $H\in\mathcal{H}$ where $K\leq H$.

For (iii), the reverse direction follows from (ii).  For the forward direction, 
let $K$ be a direct $\Ops$-factor of $\langle \mathcal{H}-\mathfrak{X}\rangle$.  
Because $\mathfrak{X}$ is closed to 
direct $\Ops$-factors, if $K\in\mathfrak{X}$ then so is every directly 
$\Ops$-indecomposable direct $\Ops$-factor of $K$, and so we insist that $K$ is
directly $\Ops$-indecomposable.  Therefore $K$ lies in a Remak $\Ops$-decomposition
of $\langle \mathcal{H}-\mathfrak{X}\rangle$.  Let $\mathcal{R}$ be
a Remak $\Ops$-decomposition of $\langle \mathcal{H}-\mathfrak{X}\rangle$
which refines $\mathcal{H}-\mathfrak{X}$.  By \thmref{thm:KRS} there
is a $\varphi\in \Aut_{\Ops\cup G} \langle \mathcal{H}-\mathfrak{X}\rangle$ 
such that $K\varphi\in \mathcal{R}$ and so $K\varphi$ is a direct 
$\Ops$-factor of the unique $H\in\mathcal{H}$ where $K\varphi\leq H$.  
As $\mathcal{H}$ is $\mathfrak{X}$-separated
and $K\varphi$ is a direct $\Ops$-factor of $H\in\mathcal{H}$, it follows
that $K\varphi \notin\mathfrak{X}$.  Thus, $K\notin\mathfrak{X}$ and 
$\{\langle\mathcal{H}-\mathfrak{X}\rangle,
\langle\mathcal{H}\cap\mathfrak{X}\rangle\}$ is $\mathfrak{X}$-separated.

For (iv), note that elements of a Remak $\Ops$-decomposition have no proper
direct $\Ops$-factors.

Finally for (v), let $\mathcal{R}$ and $\mathcal{T}$ be a Remak $\Ops$-decompositions 
of $G$ which refine $\mathcal{H}$ and $\mathcal{K}$ respectively.
Set $\mathcal{U}=\{R\in\mathcal{R}: R\leq \langle \mathcal{H}\cap \mathfrak{X}\rangle\}$.
By \thmref{thm:KRS} there is a $\varphi\in\Aut_{\Ops\cup G} G$ such that
$\mathcal{U}\varphi\subseteq \mathcal{T}$ and $\mathcal{R}\varphi
=(\mathcal{R}-\mathcal{U})\sqcup \mathcal{U}\varphi$.  As $\mathfrak{X}$ is closed 
to isomorphisms, it follows that $\mathcal{U}\varphi\subseteq\mathcal{T}\cap\mathfrak{X}$.  
As $\mathcal{H}$ is $\mathfrak{X}$-separated, $\mathcal{U}=\mathcal{R}\cap\mathfrak{X}$.
As $\Aut_{\Ops\cup G} G$ is transitive on the set of all Remak $\Ops$-decompositions
of $G$ (\thmref{thm:KRS}), we have that 
$|\mathcal{T}\cap\mathfrak{X}|=|\mathcal{R}\cap\mathfrak{X}|=|\mathcal{U}\varphi|$.
In particular, $\mathcal{U}\varphi=\mathcal{T}\cap\mathfrak{X}=
\{T\in\mathcal{T}: T\leq \langle\mathcal{K}\cap\mathfrak{X}\rangle\}$.  Hence,
$\mathcal{R}\varphi$ refines $(\mathcal{H}-\mathfrak{X})\sqcup (\mathcal{K}\cap \mathfrak{X})$
and so the latter is a direct $\Ops$-decomposition.
\end{proof}

\subsection{Up grades and down grades}\label{sec:direct-ext}
Here we introduce a companion subgroup to a direct class $\mathfrak{X}$ of $\Ops$-groups.  
These groups specify the kernels we consider in the problems of extending and lifting
in concrete settings.

\begin{defn}\label{def:grader}
An \emph{up $\Ops$-grader} (resp. \emph{down $\Ops$-grader}) for a direct class $\mathfrak{X}$ 
of $\Ops$-groups is a function $G\mapsto \mathfrak{X}(G)$ of finite $\Ops$-groups $G$ where 
$\mathfrak{X}(G)\in\mathfrak{X}$ (resp. $G/\mathfrak{X}(G)\in\mathfrak{X}$) and such that the 
following hold.
\begin{enumerate}[(i)]
	\item If $G\in\mathfrak{X}$ then $\mathfrak{X}(G)=G$ (resp. $\mathfrak{X}(G)=1$).
	\item $\mathfrak{X}(G)$ is an $\Ops$-graded subgroup of $G$.
	\item For direct $\Ops$-factor $H$ of $G$, $\mathfrak{X}(H)=H\cap \mathfrak{X}(G)$.
\end{enumerate}
The pair $(\mathfrak{X},G\mapsto \mathfrak{X}(G))$ is an up/down \emph{$\Ops$-grading pair}.
\end{defn}

If $(\mathfrak{X},G\mapsto\mathfrak{X}(G))$ is an $\Ops$-grading pair then we have 
$\mathfrak{X}(H\times K)=\mathfrak{X}(H)\times \mathfrak{X}(K)$.
First we concentrate on general and useful instances of grading pairs.
\begin{prop}\label{prop:V-inter-1}
The marginal subgroup of a variety of $\Ops$-groups is an up $\Ops$-grader and the 
verbal subgroup is a down $\Ops$-grader for the variety.
\end{prop}
\begin{proof} 
Let $\mathfrak{V}=\mathfrak{V}^{\Ops}$ be a variety of $\Ops$-groups with
defining laws $\WW$ and fix an $\Ops$-group $G$.  As the marginal function is idempotent,
\eqref{eq:marginal-closure} implies that $\mathfrak{V}^*(G)\in\mathfrak{V}$ and 
that if $G\in\mathfrak{V}$ then $G=\mathfrak{V}^*(G)$.  Similarly,
verbal subgroups are radical so that by \eqref{eq:verbal-closure} we have
$G/\mathfrak{V}(G)\in\mathfrak{V}$ and when $G\in\mathfrak{V}$ then $\mathfrak{V}(G)=1$.  
It remains to show properties (ii) and (iii) of \defref{def:grader}.

Fix a direct $\Ops$-decomposition $\mathcal{H}$ of $G$, fix an $H\in\mathcal{H}$, and 
set $K=\langle\mathcal{H}-\{H\}\rangle$.
For each $f\in G^{\XX}=(H\times K)^{\XX}$ there are unique $f_H\in H^{\XX}$  and $f_K\in K^{\XX}$
such that $f=f_H f_K$.
  Thus, for all $w\in \WW$, $w(f)=w(f_H)w(f_K)$ and so
$w(H\times K)=w(H)\times w(K)$. Hence, $\mathfrak{V}(H\times K)=\mathfrak{V}(H)\times
\mathfrak{V}(K)$.  By induction
on $|\mathcal{H}|$, $\mathcal{H}\cap\mathfrak{V}(G)=\{\mathfrak{V}(H):H\in\mathcal{H}\}$ is 
a direct $\Ops$-decomposition of $\mathfrak{V}(G)$.  So $\mathfrak{V}(G)$ is a down $\Ops$-grader.

For the marginal case, for all $f'\in \langle (h,k)\rangle^{\XX}\leq (H\times K)^{\XX}=G^{\XX}$ and 
all $f\in G^{\XX}$, again there exist unique $f_H,f'_H\in H^{\XX}$ and $f_K,f'_K\in K^{\XX}$ such that $f=f_H f_K$
and $f'=f'_H f'_K$.
Also, $w(f f')=w(f)$ if, and only if, $w(f_H f'_H)=w(f_H)$ and $w(f_K f'_K)=w(f_K)$.  Thus,
$w^*(H\times K)=w^*(H)\times w^*(K)$.  Hence, $\mathfrak{V}^*(H\times K)
=\mathfrak{V}^*(H)\times \mathfrak{V}^*(K)$ and by induction
$\mathcal{H}\cap\mathfrak{V}^*(G)$ is a direct $\Ops$-decomposition of $\mathfrak{V}^*(G)$.
Thus, $\mathfrak{V}^*(G)$ is an up $\Ops$-grader.
\end{proof}

\begin{remark}
There are examples of infinite direct decompositions $\mathcal{H}$ of infinite groups $G$ 
and varieties $\mathfrak{V}$, where $\mathfrak{V}(G)\neq \langle \mathcal{H}\cap \mathfrak{V}(G)\rangle$
\cite{Asmanov}.  
However, our definition of grading purposefully avoids infinite direct decompositions.
\end{remark}

With \propref{prop:V-inter-1} we get a simultaneous proof of some individually evident examples
of direct ascenders and descenders.

\begin{coro}\label{coro:canonical-graders}
Following the notation of \exref{ex:varieties} we have the following.
\begin{enumerate}[(i)]
\item The class $\mathfrak{N}_c$ of nilpotent groups of class at most $c$ is a direct class
with up grader $G\mapsto \zeta_c(G)$ and down grader $G\mapsto \gamma_c(G)$.
\item The class $\mathfrak{S}_d$ of solvable groups of derived length at most $d$ is a direct class
with up grader $G\mapsto (\delta_d)^*(G)$ and down grader $G\mapsto G^{(d)}$.
\item For each prime $p$ the class $\mathfrak{V}([x,y]z^p)$ of elementary abelian $p$-groups is a direct class with up grader $G\mapsto \Omega_1(\zeta_1(G))$ and down grader $G\mapsto [G,G]\mho_1(G)$.\footnote{Here $\Omega_1(X)=\langle x\in X: x^p=1\rangle$ and $\mho_1(X)=\langle x^p : x\in G\rangle$, which are traditional notations having nothing to do with our use of $\Ops$ for operators elsewhere.}
\end{enumerate}
\end{coro}

We also wish to include direct classes  $\mathfrak{N}:=\bigcup_{c\in\mathbb{N}} \mathfrak{N}_c$
and $\mathfrak{S}:=\bigcup_{d\in\mathbb{N}} \mathfrak{S}_d$.  These classes are not varieties (they are not closed to infinite direct products as required by \thmref{thm:BK}).  Therefore, we must consider alternatives to verbal and marginal groups for appropriate graders.  Our approach mimics the definitions $G\mapsto O_p(G)$ and $G\mapsto O^p(G)$.  We explain the up grader case solely.

\begin{defn}
For a class $\mathfrak{X}$, the $\mathfrak{X}$-core, $O_{\mathfrak{X}}(G)$, of a 
finite group $G$ is 
the intersection of all maximal $(\Ops\cup G)$-subgroups contained
in $\mathfrak{X}$.
\end{defn}

If $\mathfrak{V}$ is a union of a chain $\mathfrak{V}_0\subseteq \mathfrak{V}_1\subseteq\cdots $ of varieties then $1\in\mathfrak{V}$, and so the maximal $(\Ops\cup G)$-subgroups
of a group $G$ contained in $\mathfrak{V}$ is nonempty.  Also $\mathfrak{V}$ is 
closed to subgroups so that $O_{\mathfrak{V}}(G)\in \mathfrak{V}$. 
\begin{ex}\label{ex:cores}
\begin{enumerate}[(i)]
\item $O_{\mathfrak{A}}(G)$ is the intersection of all maximal 
normal abelian subgroups of $G$.  Generally there can be any 
number of maximal normal abelian subgroups of $G$ so 
$O_{\mathfrak{A}}(G)$ is not a trivial intersection.

\item $O_{\mathfrak{N}_c}(G)$ is the intersection of 
all maximal normal nilpotent subgroups of $G$ with class at most $c$.  
As in (i), this need not be a trivial intersection.  However, if
$c>\log |G|$ then all nilpotent subgroups of $G$ have class at 
most $c$ and therefore $O_{\mathfrak{N}}(G)=O_{\mathfrak{N}_c}(G)$ is the Fitting 
subgroup of $G$: the unique maximal normal nilpotent subgroup of $G$.

\item $O_{\mathfrak{S}_d}(G)$, $d>\log |G|$, 
is the unique maximal normal solvable subgroup of $G$, i.e.:
the solvable radical $O_{\mathfrak{S}}(G)$ of $G$.
\end{enumerate}
\end{ex}

\begin{lemma}\label{lem:margin-join}
Let $\mathfrak{V}$ be a group variety of $\Ops$-groups and 
$G$ an $\Ops$-group.  If $H$ is a $\mathfrak{V}$-subgroup of $G$ then
so is $\mathfrak{V}^*(G)H$, that is: $\mathfrak{V}^*(G)H\in\mathfrak{V}$.
\end{lemma}
\begin{proof}
Let $\WW$ be a set of defining laws for $\mathfrak{V}$.
Let $f'\in G^{\XX}$ with $\im f\subseteq \mathfrak{V}^*(G) H$.  Thus,
for all $w\in \WW$, there is a decomposition $f=f' f''$ where $\im f'\subseteq w^*(G)$
and $\im f''\subseteq H$.  As $w^*(G)$ is marginal to $G$ it is marginal to $H$ and so
$w(f)=w(f'')$.  As $H\in\mathfrak{V}$, $w(f'')=1$.  Thus, $w(f)=1$ and so
$w(w^*(G)H)=1$.  It follows that $\mathfrak{V}^*(G)H\in\mathfrak{V}$.
\end{proof}

\begin{prop}\label{prop:margin-core}
If $\mathfrak{V}$ is a group variety of $\Ops$-groups and $G$
an $\Ops$-group, then
\begin{enumerate}[(i)]
\item $\mathfrak{V}^*(G)\leq O_{\mathfrak{V}}(G)$, and
\item if $M$ is an $(\Ops\cup G)$-subgroup  then 
$O_{\mathfrak{V}}(G)O_{\mathfrak{V}}(M)$ 
is an $(\Ops\cup G)$-subgroup contained in $\mathfrak{V}$.
\end{enumerate}
\end{prop}
\begin{proof}
$(i)$.  By \lemref{lem:margin-join}, every maximal normal
$\mathfrak{V}$-subgroup of $G$ contains $\mathfrak{V}^*(G)$.

$(ii)$.  As $M\normaleq G$ and $O_{\mathfrak{V}}(M)$ is characteristic
in $M$, it follows that $O_{\mathfrak{V}}(M)$ is a normal 
$\mathfrak{V}$-subgroup of $G$.  Thus, $O_{\mathfrak{V}}(M)$ lies
in a maximal normal $\mathfrak{V}$-subgroup $N$ of $G$.  As 
$O_{\mathfrak{V}}(G)\leq N$ we have $O_{\mathfrak{V}}(G)O_{\mathfrak{V}}(M)
\leq N\in\mathfrak{V}$.  As $\mathfrak{V}$ is closed to subgroups,
it follows that $O_{\mathfrak{V}}(G)O_{\mathfrak{V}}(M)$ is
in $\mathfrak{V}$. 
\end{proof}

\begin{remark}
It is possible to have $\mathfrak{V}^*(G)<O_{\mathfrak{V}}(G)$.  For instance,
with $G=S_3\times C_2$ and the class $\mathfrak{A}$ of abelian groups,
the $\mathfrak{A}$-marginal subgroup is the center $1\times C_2$, whereas the 
$\mathfrak{A}$-core is $C_3\times C_2$.  
\end{remark}

\begin{prop}\label{prop:V-inter-core}
Let $G$ be a finite group with a direct decomposition $\mathcal{H}$.  If $\mathfrak{V}$ is a group variety then
\begin{equation*}
	\mathcal{H}\intersect O_{\mathfrak{V}}(G)
		=\{O_{\mathfrak{V}}(H): H\in\mathcal{H}\}
\end{equation*}
and this is a direct decomposition of $O_{\mathfrak{V}}(G)$.  In particular,
$G\mapsto O_{\mathfrak{V}}(G)$ is an up $\Ops$-grader.   Furthermore, if $\mathfrak{V}$ is a union of a chain $\mathfrak{V}_0\subseteq \mathfrak{V}_1\subseteq\cdots $ of group varieties then $O_{\mathfrak{V}}(G)$
is an up $\Ops$-grader.
\end{prop}
\begin{proof}  
Let $H\in \mathcal{H}$ and $K:=\langle \mathcal{H}-\{H\}\rangle$.  
Let $M$ be a maximal normal $\mathfrak{V}$-subgroup of $G=H\times K$.  
Let $M_H$ be the projection of $M$ to the $H$-component.  As 
$\mathfrak{V}$ is closed to homomorphic images, 
$M_H\in\mathfrak{V}$.  Furthermore, $M_H\normaleq H$ so there 
is a maximal normal $\mathfrak{V}$-subgroup $N$ of $H$ such that 
$M_H\leq N$. 

We claim that $MN\in\mathfrak{V}$.  

As $G=H\times K$, every $g\in M$ has the unique form $g=hk$, 
$h\in H$, $k\in K$.  As $M_H$ is the projection of $M$ to $H$, 
$h\in M_H\leq N$.  Thus, $g,h\in MN$ so $k\in MN$.  Thus, 
$MN=N\times M_K$, where $M_K$ is the projection of $M$ to $K$.
Now let $\mathfrak{V}=\mathfrak{V}(w)$.
For each $f:X\to MN$, write $f=f_N \times f_K$ where $f_N:X\to N$ and 
$f_K:X\to M_K$.  Hence, $w(f)=w(f_N \times f_K)=w(f_N)\times w(f_K)$.  
However, $w(N)=1$ and $w(M_K)=1$ as $N,M_K\in\mathfrak{V}$.  Thus, 
$w(f)=1$, which proves that $w(MN)=1$.  So $MN\in\mathfrak{V}$
as claimed.

As $M$ is a maximal normal $\mathfrak{V}$-subgroup of $G$, 
$M=MN$ and $N=M_H$.  Hence, $H\intersect M=N$ is a maximal normal 
$\mathfrak{V}$-subgroup of $H$.  So we have characterized the 
maximal normal $\mathfrak{V}$-subgroups of $G$ as
the direct products of maximal normal $\mathfrak{V}$-subgroups of
members  $H\in\mathcal{H}$.
Thus, $\mathcal{H}\intersect O_{\mathfrak{V}}(G)
=\{O_{\mathfrak{V}}(H) : H\in\mathcal{H}\}$ and this generates 
$O_{\mathfrak{V}}(G)$.  By \lemref{lem:induced}, $\mathcal{H}\intersect
O_{\mathfrak{V}}(G)$ is a direct decomposition of $O_{\mathfrak{V}}(G)$.
\end{proof}

\begin{coro}\label{coro:canonical-grader-II}
\begin{enumerate}[(i)]
\item The class $\mathfrak{N}$ of nilpotent groups is a direct class and
$G\mapsto O_{\mathfrak{N}}(G)$ (the Fitting subgroup) is up grader.

\item The class $\mathfrak{S}$ of solvable groups is a direct class and
$G\mapsto O_{\mathfrak{S}}(G)$ (the solvable radical) is an up grader.
\end{enumerate}
\end{coro}
\begin{proof}
For a finite group $G$, the Fitting subgroup is the $\mathfrak{N}_c$-core
where $c>|G|$.  Likewise, the solvable radical is the $\mathfrak{S}_c$-core
for $d>|G|$.  The rest follows from \propref{prop:V-inter-core}.
\end{proof}

We now turn our attention away from examples of grading pairs and
focus on their uses. In particular it is for the following ``local-global'' property 
which clarifies, in the up grader case, when a direct factor of a subgroup is also a direct factor 
of the whole group. 

\begin{prop}\label{prop:extendable}
Let $G\mapsto \mathfrak{X}(G)$ be an up $\Ops$-grader for a direct class $\mathfrak{X}$ of $\Ops$-groups
and let $G$ be an $\Ops$-group.
If $H$ is an $(\Ops\cup G)$-subgroup of $G$ and the following hold:
\begin{enumerate}[(a)]
\item for some direct $\Ops$-factor $R$ of $G$, $H\mathfrak{X}(G)=R\mathfrak{X}(G)>\mathfrak{X}(G)$, and
\item $H$ 
lies in an $\mathfrak{X}$-separated direct $(\Ops\cup G)$-decomposition
of $H\mathfrak{X}(G)$;
\end{enumerate}
then $H$ is a direct $\Ops$-factor of $G$.
\end{prop}
\begin{proof}
By (a) there is a direct $(\Ops\cup G)$-complement $C$ in $G$ to $R$.
Also $\mathfrak{X}(G)=\mathfrak{X}(R)\times \mathfrak{X}(C)$, 
as $\mathfrak{X}(G)$ is $\Ops$-graded.  Hence, $R\mathfrak{X}(G)=R\times \mathfrak{X}(C)$.
By (b), there is an $\mathfrak{X}$-separated direct $\Ops$-decomposition
$\mathcal{H}$ of $H\mathfrak{X}(G)$ such that $H\in\mathcal{H}$.  As $H\mathfrak{X}(G)>\mathfrak{X}(G)$
it follows that $H\notin\mathfrak{X}$ and so by \lemref{lem:induced}(iii), 
$\mathcal{H}-\mathfrak{X}=\{H\}$ and $X=\langle\mathcal{H}\cap\mathfrak{X}\rangle\in \mathfrak{X}$.
So 
	$$R\times \mathfrak{X}(C)=R\mathfrak{X}(G)=H\mathfrak{X}(G)=H\times X.$$
Let $\mathcal{A}$ be Remak $(\Ops\cup G)$-decomposition of $R$.  Since $\mathfrak{X}(C)\in\mathfrak{X}$,
$\mathcal{A}\sqcup\{\mathfrak{X}(C)\}$ is an $\mathfrak{X}$-separated direct $(\Ops\cup G)$-decomposition
of $R\mathfrak{X}(G)$.  By \propref{prop:direct-class}(v),
	$$\mathcal{C}=\{H\}\sqcup \{\mathfrak{X}(C)\}\sqcup (\mathcal{A}\cap\mathfrak{X})$$ 
is an $\mathfrak{X}$-separated direct $(\Ops\cup G)$-decomposition of $R\mathfrak{X}(G)$,
and we note that $\{H\}=\mathcal{C}-\mathfrak{X}$.  We claim that
$\{H,C\}\sqcup(\mathcal{A}\cup \mathfrak{X})$ is a direct $\Ops$-decomposition of $G$.
Indeed, $H\cap \langle C,\mathcal{A}\cap \mathfrak{X}\rangle\leq R\mathfrak{X}(G)\cap C\mathfrak{X}(G)
=\mathfrak{X}(G)$ and so $H\cap \langle C,\mathcal{A}\cap \mathfrak{X}\rangle
=H\cap \langle \mathfrak{X}(C),\mathcal{A}\cap \mathfrak{X}\rangle=1$.
Also, $\mathfrak{X}(C)\leq \langle H,C,\mathcal{A}\cap\mathfrak{X}\rangle$ thus 
$\langle H,C,\mathcal{A}\cap\mathfrak{X}\rangle=G$.  As the members of
$\{H,C\}\sqcup(\mathcal{A}\cap\mathfrak{X})$ are $(\Ops\cup G)$-subgroups we have proved the claim.
In particular, $H$ is a direct $\Ops$-factor of $G$.
\end{proof}

\subsection{Direct chains}\label{sec:chains}
In \thmref{thm:Lift-Extend} we specified conditions under which any direct decomposition of
an appropriate subgroup, resp. quotient, led to a solution of the extension (resp. lifting) problem.
However, within that theorem we see that it is not the direct decomposition of the subgroup
(resp. quotient group) which can be extended (resp. lifted).  Instead it a some unique
partition of the direct decomposition.  Finding the correct partition by trial and error is
an exponentially sized problem.  To avoid this we outline a data structure which enables a 
greedy algorithm to find this unique partition.  The algorithm itself is given in Section
\ref{sec:merge}.  The key result of this section is \thmref{thm:chain}.

Throughout this section we suppose that $G\to \mathfrak{X}(G)$ is an (up) $\Ops$-grader for a 
direct class $\mathfrak{X}$.
\begin{defn}\label{def:chain}
A \emph{direct chain} is a proper chain $\mathcal{L}$ of $(\Ops\cup G)$-subgroups starting at $\mathfrak{X}(G)$ and ending at $G$, and where there is a direct $\Ops$-decomposition $\mathcal{R}$ of $G$ with:
\begin{enumerate}[(i)]
\item for all $L\in\mathcal{L}$, $L=\langle\mathcal{R}\cap L\rangle$, and

\item for each $L\in\mathcal{L}-\{G\}$, there is a unique $R\in\mathcal{R}$
such that the successor $M\in\mathcal{L}$ to $L$ satisfies:
$R\mathfrak{X}(G)\cap L\neq R\mathfrak{X}(G)\cap M$.  We call $R$ the \emph{direction of $L$}.
\end{enumerate}
We call $\mathcal{R}$ a set of directions for $\mathcal{L}$.
\end{defn}

If $\mathcal{L}$ is a direct chain with directions $\mathcal{R}$, then
for all $L\in\mathcal{L}$, $\mathcal{R}\cap L$ is a direct $\Ops$-decomposition of $L$ 
(\lemref{lem:induced}(i)).  When working with direct chains it helps to remember that 
for all $(\Ops\cup G)$-subgroups $L$ and $R$ of $G$, if $\mathfrak{X}(G)\leq L$, then 
$(R\cap L)\mathfrak{X}(G)=R\mathfrak{X}(G)\cap L$.  Also, if $\mathfrak{X}(G)\leq L< M\leq G$,
$L=\langle\mathcal{R}\cap L\rangle$ and $M=\langle\mathcal{R}\cap M\rangle$, and
\begin{equation}\label{eq:unique-direction}
\forall R\in\mathcal{R}-\mathfrak{X},\qquad
R\mathfrak{X}(G)\cap L=R\mathfrak{X}(G)\cap M
\end{equation}
then $L=\langle\mathcal{R}\cap L\rangle=\langle\mathcal{R}\cap L,\mathfrak{X}(G)\rangle
	=\langle\mathcal{R}\cap M,\mathfrak{X}(G)\rangle=\langle\mathcal{R}\cap M\rangle=M$.
Therefore, it suffices 
to show there is
at most one $R\in\mathcal{R}-\mathfrak{X}$ such that $R\mathfrak{X}(G)\cap L\neq
R\mathfrak{X}(G)\cap M$.  

\begin{lemma}\label{lem:cap}
Suppose that $\mathcal{H}=\mathcal{H}\mathfrak{X}(G)$ is an $(\Ops\cup G)$-decomposition of $G$ 
such that $\mathcal{H}$ refines $\mathcal{R}\mathfrak{X}(G)$, for a direct $\Ops$-decomposition $\mathcal{R}$.  It follows that, 
if $L=\langle\mathcal{J},\mathfrak{X}(G)\rangle$, for some $\mathcal{J}\subseteq \mathcal{H}$, 
then $L=\langle \mathcal{R}\cap L\rangle$.
\end{lemma}
\begin{proof}
As $\mathfrak{X}(G)\leq L$, for each $R\in\mathcal{R}$, $R\cap \mathfrak{X}(G)\leq R\cap L$.
As $\mathfrak{X}(G)$ is $(\Ops\cup G)$-graded, $\mathfrak{X}(G)=\langle\mathcal{R}\cap \mathfrak{X}(G)\rangle$.
Thus, $\mathfrak{X}(G)\leq \langle \mathcal{R}\cap L\rangle$.
Also, $\mathcal{H}$ refines $\mathcal{R}\mathfrak{X}(G)$.  Thus, for each
$J\in\mathcal{J}\subseteq \mathcal{H}$ there is a unique 
$R\in\mathcal{R}-\{R\in\mathcal{R}: R\leq \mathfrak{X}(G)\}$ such that $J\leq R\mathfrak{X}(G)$.
As $L=\langle\mathcal{J},\mathfrak{X}(G)\rangle$, $J\leq L$ and so $J\leq R\mathfrak{X}(G)\cap L=(R\cap L)\mathfrak{X}(G)$.
Now $R\cap L,\mathfrak{X}(G)\leq \langle \mathcal{R}\cap L\rangle$ thus 
$J\leq \langle\mathcal{R}\cap L\rangle$.  Hence $L=\langle\mathcal{J},\mathfrak{X}(G)\rangle\leq 
\langle\mathcal{R}\cap L\rangle\leq L$.
\end{proof}

\begin{lemma}\label{lem:drop-H}
If $\mathcal{H}$ is an $(\Ops\cup G)$-decomposition of $G$ and 
$\mathcal{R}$ a direct $(\Ops\cup G)$-decomposition of $G$ such that
$\mathcal{H}=\mathcal{H}\mathfrak{X}(G)$ refines $\mathcal{R}\mathfrak{X}(G)$,
then for all $\mathcal{J}\subset\mathcal{H}$ and all $H\in\mathcal{H}-\mathcal{J}$,
there is a unique $R\in\mathcal{R}$ such that $H\leq R\mathfrak{X}(G)$
 and
$$\langle\mathcal{R}-\{R\}\rangle \mathfrak{X}(G)\cap \langle H,\mathcal{J},\mathfrak{X}(G)\rangle
=\langle\mathcal{R}-\{R\}\rangle\mathfrak{X}(G)\cap \langle\mathcal{J},\mathfrak{X}(G)\rangle.$$
\end{lemma}
\begin{proof}  
Fix $\mathcal{J}\subseteq\mathcal{H}$ and $H\in\mathcal{H}-\mathcal{J}$.  
By the definition of refinement there is a unique $R\in\mathcal{R}$ such that
$H\leq R\mathfrak{X}(G)$.  Set $J=\langle\mathcal{J},\mathfrak{X}(G)\rangle$ and 
$C=\langle \mathcal{R}-\{R\}\rangle$.
By \lemref{lem:cap}, $\mathcal{R}\cap HJ$ and $\mathcal{R}\cap J$ are 
direct $(\Ops\cup G)$-decompositions of $HJ$ and $J$ respectively.  
As $J=(R\cap J)\times (C\cap J)$ and $\mathfrak{X}(G)\leq J$, we get that
$J=(R\mathfrak{X}(G)\cap J)(C\mathfrak{X}(G)\cap J)$.
Also, $\mathfrak{X}(G)$ is $(\Ops\cup G)$-graded; hence, by 
\lemref{lem:induced}(ii), $G/\mathfrak{X}(G)=R\mathfrak{X}(G)/\mathfrak{X}(G)
\times C\mathfrak{X}(G)/\mathfrak{X}(G)$ and $C\mathfrak{X}(G)\cap R\mathfrak{X}(G)=
\mathfrak{X}(G)$.

Combining  the modular law with $\mathfrak{X}(G)\leq H\leq R\mathfrak{X}(G)$ and 
$R\mathfrak{X}(G)\cap C\mathfrak{X}(G)=\mathfrak{X}(G)$ 
we have that
\begin{align*}
	C\mathfrak{X}(G)\cap HJ
		& = C\mathfrak{X}(G) \cap \Big (H(R\mathfrak{X}(G)\cap J)\cdot
				(C\mathfrak{X}(G)\cap J)\Big)\\
		& = \Big(C\mathfrak{X}(G)\cap H(R\mathfrak{X}(G)\cap J)\Big) 
				( C\mathfrak{X}(G)\cap J)\\
		& = (C\mathfrak{X}(G)\cap R\mathfrak{X}(G)\cap HJ)(C\mathfrak{X}(G)\cap J) \\
		& = \mathfrak{X}(G)(C\mathfrak{X}(G)\cap J)=C\mathfrak{X}(G)\cap J.
\end{align*}
Thus, $C\mathfrak{X}(G)\cap HJ=C\mathfrak{X}(G)\cap J$.
\end{proof}

\begin{prop}\label{prop:chain-chain}
If $\mathcal{H}=\mathcal{H}\mathfrak{X}(G)$ is an $(\Ops\cup G)$-decomposition of $G$ and
$\mathcal{R}$ is a direct $\Ops$-decomposition of $G$ such that $\mathcal{H}$
refines $\mathcal{R}\mathfrak{X}(G)$, then every maximal proper chain $\mathscr{C}$ of
subsets of $\mathcal{H}$ induces a direct chain $\{\langle \mathcal{C},\mathfrak{X}(G)\rangle :
\mathcal{C}\in\mathscr{C}\}$.
\end{prop}
\begin{proof}
For each $\mathcal{C}\subseteq\mathcal{H}$, by \lemref{lem:cap}, 
$\langle\mathcal{C}\rangle=\big\langle\mathcal{R}\cap \langle\mathcal{C}\rangle\big\rangle$. 
The rest follows from \lemref{lem:drop-H}.
\end{proof}

The following \thmref{thm:chain} is a critical component of the proof of the algorithm for \thmref{thm:FindRemak}, specifically in proving \thmref{thm:merge}.  What it says is that we can proceed through any direct chain as the $\mathfrak{X}$-separated direct decompositions of lower terms in the chain induce direct factors of the next term in the chain, and in a predictable manner.
\begin{thm}\label{thm:chain}
If $\mathcal{L}$ is a direct chain with directions $\mathcal{R}$, 
$L\in\mathcal{L}-\{G\}$, and $R\in\mathcal{R}$ is 
the direction of $L$, then for every 
$\mathfrak{X}$-separated direct $(\Ops\cup G)$-decomposition $\mathcal{K}$ of $L$ 
such that $\mathcal{K}\mathfrak{X}(G)$ refines $\mathcal{R}\mathfrak{X}(G)\cap L$, 
it follows that 
$$\big\{K\in\mathcal{K}-\mathfrak{X}: K\leq 
	\langle \mathcal{R}-\{R\}\rangle \mathfrak{X}(G)\big\}$$
lies in an $\mathfrak{X}$-separated direct $(\Ops\cup G)$-decomposition of the successor
to $L$.
\end{thm}
\begin{proof}  
Let $M$ be the successor to $L$ in $\mathcal{L}$ and set $C=\langle\mathcal{R}-\{R\}\rangle$.  
As $\mathcal{K}\mathfrak{X}(G)$ refines 
$\mathcal{R}\mathfrak{X}(G)\cap L$, it also refines 
$\{R\mathfrak{X}(G)\cap L,C\mathfrak{X}(G)\cap L\}$
and so
\begin{align*}
	C\mathfrak{X}(G)\cap L
		& = \langle K\in\mathcal{K}, K\leq C\mathfrak{X}(G)\rangle
		 = \langle K\in\mathcal{K}-\mathfrak{X}, K\leq C\mathfrak{X}(G)\rangle\mathfrak{X}(G).
\end{align*} 
Since $\mathcal{K}$ is $\mathfrak{X}$-separated $F=\langle K\in\mathcal{K}-\mathfrak{X}, 
K\leq C\mathfrak{X}(G)\rangle$ has no direct $(\Ops\cup G)$-factor in $\mathfrak{X}$.  Also, 
as the direction of $L$ is $R$, $C\mathfrak{X}(G)\cap M=C\mathfrak{X}(G)\cap L$ and so
\begin{align*}
(C\cap M)\mathfrak{X}(G) 
	& = C\mathfrak{X}(G)\cap M\\
	& = C\mathfrak{X}(G)\cap L\\
	& = \langle K\in\mathcal{K}-\mathfrak{X}, K\leq C\mathfrak{X}(G)\rangle \mathfrak{X}(G)\\
	& =F \times \langle \mathcal{K}\cap \mathfrak{X}\rangle.
\end{align*}
Using $(M,F,C\cap M)$ in the role of $(G,H,R)$ in \propref{prop:extendable}, it 
follows that $F$ is a direct $(\Ops\cup G)$-factor of $M$.  
In particular, $\{K\in\mathcal{K}-\mathfrak{X}, K\leq C\mathfrak{X}(G)\}$ 
lies in a direct $(\Ops\cup G)$-decomposition of $M$.
\end{proof}

\section{Algorithms to lift, extend, and match direct decompositions}\label{sec:lift-ext-algo}

Here we transition into algorithms beginning with a small modification of a technique introduced
by Luks and Wright to find a direct complement to a direct factor (\thmref{thm:FindComp-invariant}).  We then produce an algorithm
\textalgo{Merge} (\thmref{thm:merge}) to lift direct decompositions for appropriate quotients. 
That algorithm is the work-horse which glues together the unique constituents predicted by \thmref{thm:Lift-Extend}.  That task asks us to locate a unique partition of a certain set, but in a manner that does not test each of the exponentially many partitions.  The proof relies heavily on results such as \thmref{thm:chain} to prove that an essentially greedy algorithm will suffice.  

For brevity we have opted to describe the algorithms only for the case of lifting decompositions.  The natural duality of up and down graders makes it possible to modify the methods to prove similar results for extending decompositions.

This section assumes familiarity with Sections \ref{sec:tools} and \ref{sec:lift-ext}.

\subsection{Constructing direct complements}\label{sec:complements}

In this section we solve the following problem in polynomial-time.

\begin{prob}{\sc Direct-$\Ops$-Complement}
\begin{description}
\item[Given] a $\Ops$-group $G$ and an $\Ops$-subgroup $H$,
\item[Return] an $\Ops$-subgroup $K$ of $G$ such that $G=H\times K$,
or certify that no such $K$ exists.
\end{description}
\end{prob}

Luks and Wright gave independent solutions to 
{\sc Direct-$\emptyset$-Complement} in
 back-to-back lectures at the University of Oregon \cite{Luks:comp,Wright:comp}.
\begin{thm}[Luks \cite{Luks:comp},Wright \cite{Wright:comp}]\label{thm:FindComp}
For groups of permutations,
{\sc Direct-$\emptyset$-Complement} has a polynomial-time solution
\end{thm} 
Both \cite{Luks:comp} and \cite{Wright:comp} reduce 
\textalgo{Direct-$\emptyset$-Complement} to the following problem
(here generalized to $\Ops$-groups):

\begin{prob}{\sc $\Ops$-Complement-Abelian}
\begin{description}
\item[Given] an $\Ops$-group $G$ and an abelian $(\Ops\cup G)$-subgroup $M$,
\item[Return] an  $\Ops$-subgroup $K$ of $G$ such that $G=M\rtimes K$,
or certify that no such $K$ exists.
\end{description}
\end{prob}

To deal with operator groups we use some modifications to the problems above.  Many of the steps
are conceived within the group $\langle \Ops \theta\rangle\ltimes G\leq \Aut G\ltimes G$.  
However, to execute these algorithms we cannot assume that $\langle \Ops\theta\rangle\ltimes G$ is a permutation
group as it is possible that these groups have no small degree permutation representations (e.g. $G=\mathbb{Z}_p^d$
and $\langle \Ops\theta\rangle=\GL(d,p)$).  Instead we operate
within $G$ and account for the action of $\Ops$ along the way.

\begin{lemma}\label{lem:pres}
Let $G$ be an $\Ops$-group where $\theta:\Ops\to \Aut G$.
If $\{\langle \XX |\RR\rangle,f,\ell\}$ is a constructive presentation 
for $G$ and $\langle \Ops|\RR'\rangle$ a presentation 
for $A:=\langle \Ops\theta\rangle\leq \Aut G$ with respect to $\theta$, then 
$\langle \Ops\sqcup \XX| \RR'\ltimes\RR\rangle$
is a presentation for $A\ltimes G$ with respect to $\theta\sqcup f$, where
\begin{equation*}
	\RR'\ltimes\RR= \RR'\sqcup\RR\sqcup
		\{(xf)^{s}\ell\cdot (x^{s})^{-1} : 
			x\in\XX,s\in\Ops\},\textnormal{ and}
\end{equation*}
\begin{equation*}
\forall z\in \Ops\sqcup \XX,\quad
	z(\theta\sqcup f)=\left\{\begin{array}{cc}  
		z\theta & z\in\Ops,\\
		zf & z\in \XX.
	\end{array}\right.
\end{equation*}
\end{lemma}
\begin{proof}  
Without loss of generality we assume 
$F(\Ops),F(\XX)\leq F(\Ops\sqcup\XX)$.  
Let $K$ be the normal closure of $\RR'\ltimes \RR $ in 
$F(\Ops\sqcup\XX)$.  For each $s\in\Ops$ and each $x\in\XX$
it follows that $Kx^{s}=K(xf)^{s}\ell\leq N=\langle K,F(\XX)\rangle$.
In particular, $N$ is normal in $F(\Ops\sqcup\XX)$.  Set 
$C=\langle K,F(\Ops)\rangle$.  It follows that $F(\Ops\sqcup\XX)
=\langle C,N\rangle=CN$.  Thus, $H=F(\Ops\sqcup\XX)/K=CN/K=(C/K)(N/K)$ 
and $N/K$ is normal in $H$.  
Since $C/K$ and $N/K$ satisfy the presentations for $A$ and $G$ respectively,
it follows that $H$ is a quotient of $A\ltimes G$.  To show that $H\cong A\ltimes G$ 
it suffices to notice that $A\ltimes G$ satisfies the relations in 
$\RR'\ltimes\RR $, with respect to
$\Ops\sqcup\XX$ and $\theta\sqcup\ell$.   Indeed,
for all $s\in \Ops$ and all $x\in\XX$ we see that
\begin{align*}
	x^{s}(\widehat{\theta\sqcup f})
		& = (s\theta^{-1},1)(1,xf)(s\theta,1) 
		 = (1,(xf)^{s})
		 = (1, (xf)^{s} \ell \hat{f})
		 = (xf)^{s} \ell(\widehat{\theta\sqcup f}),
\end{align*}
which implies that $(xf)^{s}\ell (x^{s})^{-1}\in \ker \widehat{\theta\sqcup f}$;
so, $K\leq \ker \widehat{\theta\sqcup f}$.  Hence, $\langle \Ops\sqcup \XX|
\RR'\ltimes\RR \rangle$ is a presentation for $A\ltimes G$.
\end{proof}

\begin{prop}\label{prop:InvComp}
{\sc $\Ops$-Complement-Abelian} has a polynomial-time solution.
\end{prop}
\begin{proof}
Let $M, G\in \mathbb{G}_n$, and $\theta:\Ops\to \Aut G$ a function, where $M$ is an abelian
$(\Ops\cup G)$-subgroup of $G$.

\emph{Algorithm.}
Use {\sc Presentation} to produce a constructive 
presentation $\{\langle \XX|\RR\rangle,f,\ell\}$ for $G$ mod $M$.  For each 
$s\in \Ops$ and each $x\in\XX$, define
	$$r_{s,x}=(xf^{s})\ell\cdot(x^{s})^{-1}\in F(\Ops\sqcup\XX).$$
Use {\sc  Solve} to decide if there is a
$\mu\in M^{\XX}$ where
\begin{align}\label{eq:comp-rel-1}
	\forall r\in \RR ,&\quad r(f\mu) = 1,\textnormal{ and }\\
\label{eq:comp-rel-2}
	\forall s\in\Ops,\forall x\in\XX & \quad r_{s,x}(f\mu)=1.
\end{align}
If no such $\mu$ exists, then assert that $M$ has no $\Ops$-complement in $G$; otherwise, 
return $K=\langle x(f\mu)=(xf)(x\mu) : x\in\XX\rangle$.

\emph{Correctness.}  Let $A=\langle\Ops\theta\rangle\leq \Aut G$ and let 
$\langle \Ops|\RR '\rangle$ be a presentation of $A$ with
respect to $\theta$.  The algorithm creates a constructive presentation 
$\{\langle\XX|\RR\rangle ,f,\ell\}$ for $G$ mod $M$ and so by 
\lemref{lem:pres}, $\langle\Ops\sqcup \XX|\RR '\ltimes\RR \rangle$
is a presentation for $A\ltimes G$ mod $M$ with respect to $\theta\sqcup f$.

First suppose that the algorithm returns $K=\langle x(f\mu):x\in\XX\rangle$.  
As $\XX f\subseteq KM$ we get that 
$G=\langle \XX f\rangle\leq KM\leq G$.  By \eqref{eq:comp-rel-1}, 
$r(f\mu)=1$ for all $r\in \RR $.  Therefore $K$ satisfies the defining
relations of $G/M\cong K/(K\cap M)$, which forces $K\cap M=1$ and so $G=K\ltimes M$.  
By \eqref{eq:comp-rel-1} and \eqref{eq:comp-rel-2}, the generator set 
$\Ops\theta\sqcup\{x\bar{\mu}:x\in\XX \}f$ of 
$\langle A,K\rangle$ satisfies the defining relations $\RR '\ltimes \RR $
of $(A\ltimes G)/M$ and so $\langle A,K\rangle$ is isomorphic to a quotient of 
$(A\ltimes G)/M$ where $K$ is the image of $G/M$.  This shows $K$ is normal in 
$\langle A,K\rangle$.  In particular, $\langle K^{\Ops}\rangle\leq K$.  Therefore if the algorithm 
returns a subgroup then the return is correct.

Now suppose that there is a $K\leq G$ such that $\langle K^{\Ops}\rangle\leq K$ and $G=K\ltimes M$.
We must show that in this case the algorithm returns a subgroup.
We have that $G=\langle\XX f\rangle$ and the generators $\XX f$ satisfy
(mod $M$) the relations $\RR $.   Let $\varphi:G/M\to K$
be the isomorphism $kM\varphi=k$, for all $km\in KM=G$, where $k\in K$ and $m\in M$.  
Define $\tau:\XX \to M$ by $x\tau=(xf)^{-1} (xfM)\varphi$, for all $x\in\XX $.  
Notice $\langle x(x\tau): x\in\XX \rangle=K$.  Furthermore, 
$\Phi:(a,hM)\mapsto (a,hM\varphi)$ is an isomorphism $A\ltimes(G/M)\to A\ltimes K$. 
As $\RR \subseteq F(\XX )$ it follows that
$r((\theta\sqcup f)\Phi)=r(f)\Phi=1$, for all $r\in \RR$.  Also,
\begin{align*}
\forall z\in \Ops\sqcup \XX ,& \quad
	z(\theta\sqcup f)\Phi = \left\{
	\begin{array}{cc}
		(z\theta, 1), & z\in \Ops;\\
	 	(1,(xfM)\varphi) = (1, x\bar{\tau}), & z\in \XX .
	\end{array}\right.
\end{align*}
Therefore, $r(f\tau)=r((\theta\sqcup f)\Phi)=1$ for all $r\in\RR $.  Thus, 
an appropriate $\tau\in M^{\XX }$ exists and the algorithm is guaranteed to find 
such an element and return an $\Ops$-subgroup of $G$ complementing $M$.

\emph{Timing.}  The algorithm applies two polynomial-time algorithms.
\end{proof}

\begin{thm}\label{thm:FindComp-invariant}
\textalgo{Direct-$\Ops$-Complement} has a polynomial-time solution.
\end{thm}
\begin{proof}
Let $H, G\in\mathbb{G}_n$ and $\theta:\Ops\to \Aut G$, where $\langle H^{\Ops}\rangle\leq H\leq G$.

\emph{Algorithm.}  Use {\sc Member} to determine if $H$ is an $(\Ops\cup G)$-subgroup of $G$.
If not, then this certifies that $H$ is not a direct factor of $G$.  
Otherwise, use {\sc Normal-Centralizer} to compute $C_G(H)$ and $ \zeta_1(H)$.  
Using {\sc Member}, test if $G=HC_G(H)$ and if $\langle C_G(H)^{\Ops}\rangle= C_G(H)$.  If either fails, then certify that
 $H$ is not a direct $\Ops$-factor of $G$.
Next, use \propref{prop:InvComp} to find an $\Ops$-subgroup $K\leq C_G(H)$ such that
$C_G(H)=\zeta_1(H)\rtimes K$, or determine that no such $K$ exists.  If
$K$ exists, return $K$; otherwise, $H$ is not a direct $\Ops$-factor of $G$.

\emph{Correctness.}  Note that if $G=H\times J$ is a direct $\Ops$-decomposition then
$H$ and $J$ are $(\Ops\cup G)$-subgroups of $G$, $G=HC_G(H)$, and $C_G(H)=\zeta_1(H)\times J$.  
As $\Ops\theta\subseteq \Aut G$,
$\zeta_1(H)$ is an $\Ops$-subgroup and therefore $C_G(H)$ is an $\Ops$-subgroup.  Therefore the
tests within the algorithm properly identify cases where $H$ is not a direct $\Ops$-factor of $G$.  
Finally, if the algorithm returns an $\Ops$-subgroup $K$ such that 
$C_G(H)=\zeta_1(H)\rtimes K=\zeta_1(G)\times K$, then $G=H\times K$ is a direct $\Ops$-decomposition.

\emph{Timing.}  The algorithm makes a bounded number of calls to polynomial-time algorithms.
\end{proof}

\subsection{Merge}\label{sec:merge}

In this section we provide an algorithm which given an appropriate direct decomposition of 
a quotient group produces a direct decomposition of original group.  

Throughout this section we assume that $(\mathfrak{X},G\mapsto \mathfrak{X}(G))$ is an up $\Ops$-grading pair in which $\zeta_1(G)\leq \mathfrak{X}(G)$.

The constraints of exchange by $\Aut_{\Ops\cup G} G$ given in \lemref{lem:KRS} can be sharpened 
to individual direct factors as follows.  (Note that \propref{prop:back-forth} is false when 
considering the action of $\Aut G$ on direct factors.)
\begin{prop}\label{prop:back-forth}
Let $X$ and $Y$ be direct $\Ops$-factors of $G$ with no abelian direct $\Ops$-factor.
The following are equivalent.
\begin{enumerate}[(i)]
\item $X\varphi=Y$ for some $\varphi\in \Aut_{\Ops\cup G} G$.
\item $X\zeta_1(G)=Y\zeta_1(G)$.
\end{enumerate}
\end{prop}
\begin{proof}
By \eqref{eq:central}, $\Aut_{\Ops\cup G} G$ is the
identity on $G/\zeta_1(G)$; therefore (i) implies (ii).

Next we show (ii) implies (i).  Recall that $\mathfrak{A}$ is the class of abelian groups.
Let $\{X,A\}$ and $\{Y,B\}$ be direct $\Ops$-decompositions of $G$.  Choose 
Remak $(\Ops\cup G)$-decompositions $\mathcal{R}$ and $\mathcal{C}$ 
which refine $\{X,A\}$ and $\{Y,B\}$ respectively.
Let $\mathcal{X}=\{R\in\mathcal{R}: R\leq X\}$.  By \thmref{thm:KRS} there is a
$\varphi\in\Aut_{\Ops\cup G} G$ such that $\mathcal{X}\varphi\subseteq \mathcal{C}$.
However, $\varphi$ is the identity on $G/\zeta_1(G)$.  Hence,
$\langle\mathcal{X}\rangle\zeta_1(G)=X\zeta_1(G)\varphi=Y\zeta_1(G)$.
Thus, $\mathcal{X}\varphi\subseteq \{C\in \mathcal{C}: C\leq Y\zeta_1(G)\}-\mathfrak{A}$.
Yet, $\mathcal{C}$ refines $\{Y,B\}$ and $Y$ has no direct $\Ops$-factor in $\mathfrak{A}$.
Thus, $$\{C\in \mathcal{C}: C\leq Y\zeta_1(G)=Y\times \zeta_1(B)\}-\mathfrak{A}
=\{C\in\mathcal{C}:C\leq Y\}.$$
Thus, $\mathcal{X}\varphi\subseteq\mathcal{Y}$.  By reversing the roles of $X$ and $Y$ 
we see that $\mathcal{Y}\varphi'\subseteq\mathcal{X}$ for some $\varphi'$.  Thus,
$|\mathcal{X}|=|\mathcal{Y}|$.  So we conclude that $\mathcal{X}\varphi=\mathcal{Y}$
and $X\varphi=Y$.
\end{proof}

\begin{thm}\label{thm:Extend}
There is a polynomial-time algorithm which, given an $\Ops$-group $G$ and a set
$\mathcal{K}$ of $(\Ops\cup G)$-subgroups such that
\begin{enumerate}[(a)]
\item $\mathfrak{X}(\langle\mathcal{K}\rangle)=\mathfrak{X}(G)$ and 
\item $\mathcal{K}$ is a direct $(\Ops\cup G)$-decomposition of $\langle \mathcal{K}\rangle$,
\end{enumerate}
returns a direct $\Ops$-decomposition $\mathcal{H}$ of $G$ such that 
\begin{enumerate}[(i)]
\item $|\mathcal{H}-\mathcal{K}|\leq 1$,
\item if $K\in\mathcal{K}$ such that
$\langle \mathcal{H}\cap\mathcal{K}, K\rangle$ has a direct $\Ops$-complement in $G$, then
$K\in\mathcal{H}$; and 
\item if $K\in\mathcal{K}-\mathfrak{X}$
such that $K$ is a direct $(\Ops\cup G)$-factor of $G$, then $K\in\mathcal{H}$.
\end{enumerate}
\end{thm}
\begin{proof}
\emph{Algorithm.} 
\begin{code}{Extend$(~G,~\mathcal{K}~)$}
	$\mathcal{L}=\emptyset$; $\lfloor G\rfloor = G$;\\
	\textit{/* Using the algorithm for \thmref{thm:FindComp-invariant} to determine 
				the existence of $H$, execute the following. */}\\
	\cwhile{$\exists K\in\mathcal{K}, \exists H, \mathcal{L}\sqcup\{K,H\}$ is a direct
		$\Ops$-decomposition of $G$}
	\begin{block*}
		$\lfloor G\rfloor = H$;\\
		$\mathcal{L} = \mathcal{L}\sqcup\{K\}$;\\
		$\mathcal{K} = \mathcal{K}-\{K\}$;\\
	\end{block*}
	\creturn{$\mathcal{H}=\mathcal{L}\sqcup\{\lfloor G\rfloor\}$}
\end{code}

\emph{Correctness.}  We maintain the following loop invariant (true at the start and end
of each iteration of the loop): $\mathcal{L}\sqcup\{\lfloor G\rfloor\}$ is a direct
$(\Ops\cup G)$-decomposition of $G$ and $\mathcal{L}\subseteq \mathcal{K}$. The loop exits once $\mathcal{L}\sqcup\{\lfloor G\rfloor\}$
satisfies (ii).  Hence, $\mathcal{H}=\mathcal{L}\sqcup\{\lfloor G\rfloor\}$ satisfies (i) and (ii).

For (iii), suppose that $\mathcal{K}$ is $\mathfrak{X}$-separated and that 
$K\in(\mathcal{K}-\mathfrak{X})-\mathcal{H}$ such that $K$ is a direct 
$(\Ops\cup G)$-factor of $G$.  Let $\langle F^{\Ops}\rangle\leq F\leq G$ such that $\{F,K\}$ 
is a direct $(\Ops\cup G)$-decomposition of $G$ and $\mathcal{R}$ a Remak 
$(\Ops\cup G)$-decomposition of $G$ which refines $\{F,K\}$.  Also let 
$\mathcal{T}$ be a Remak $(\Ops\cup G)$-decomposition of $G$ which refines $\mathcal{H}$.
Set $\mathcal{X}=\{R\in\mathcal{R}: R\leq K\}$, and note
that $\mathcal{X}\subseteq \mathcal{R}-\mathfrak{X}$ as $K$ has no direct $\Ops$-factor
in $\mathfrak{X}$.  By \thmref{thm:KRS} we can exchange $\mathcal{X}$ with 
a $\mathcal{Y}\subseteq \mathcal{T}-\mathfrak{X}$ to create a Remak $(\Ops\cup G)$-decomposition
$(\mathcal{T}-\mathcal{Y})\sqcup \mathcal{X}$ of $G$.  As $\zeta_1(G)\leq\mathfrak{X}(G)$
we get $\mathcal{R}\mathfrak{X}(G)=\mathcal{T}\mathfrak{X}(G)$ and 
$\mathcal{X}\mathfrak{X}(G)=\mathcal{Y}\mathfrak{X}(G)$ (\lemref{lem:KRS},
\propref{prop:back-forth}).  Thus, by (a) and then (b),
\begin{align*}
\langle\mathcal{Y}\rangle \cap 
\langle\mathcal{H}\cap \mathcal{K}\rangle
	& \equiv \langle\mathcal{X}\rangle \cap 
		\langle\mathcal{H}\cap \mathcal{K}\rangle
		& \pmod{\mathfrak{X}(G)}\\
	& \equiv K \cap \langle\mathcal{H}\cap \mathcal{K}\rangle 
		& \pmod{\mathfrak{X}(\langle \mathcal{K}\rangle)}\\
	& \leq K\cap \langle \mathcal{K}-\{K\}\rangle\\
	& \equiv 1
\end{align*}
Therefore $\langle\mathcal{Y}\rangle \leq 
\langle(\mathcal{T}-\mathfrak{X})-\{T\in\mathcal{T}:
T\leq \langle\mathcal{H}-\mathcal{K}\rangle\}\rangle$.  
Thus, 
$$\mathcal{J}=(\mathcal{H}\cap\mathcal{K})\sqcup \{K\}\sqcup \{\langle (\mathcal{T}
-\mathcal{Y})-\{T\in\mathcal{T}:T\leq \langle\mathcal{H}\cap\mathcal{K}\rangle\}$$ 
is a direct $\Ops$-decomposition of $G$ and $(\mathcal{H}\cap\mathcal{K})\sqcup\{K\}\subseteq 
\mathcal{J}\cap\mathcal{K}$ which shows that $\mathcal{L}$ is not maximal.  By the contrapositive we have (iii).

\emph{Timing.}  This loop makes $|\mathcal{K}|\leq \log_2 |G|$ calls to a polynomial-time algorithm
for \textalgo{Direct-$\Ops$-Complement}.
\end{proof}

Under the hypothesis of \thmref{thm:Extend} it is not possible to extend (iii) to 
say that if $K\in\mathcal{K}$ and $K$ is a direct $\Ops$-factor of $G$ then $K\in \mathcal{H}$.  
Consider the following example (where $\Ops=\emptyset$).
\begin{ex}
Let $G=D_8\times \mathbb{Z}_2$, $D_8=\langle a,b|a^4,b^2,(ab)^2\rangle$.  Use 
$\mathfrak{A}$ (the class of abelian groups) for $\mathfrak{X}$ and 
$\mathcal{K}=\{\langle (0,1)\rangle,\langle (a^2,1)\rangle\}$.
Each member of $\mathcal{K}$ is a direct factor of $G$, but $\mathcal{K}$ is not 
contained in any direct decomposition of $G$.
\end{ex}

\begin{lem}\label{lem:count}
If $\mathcal{K}$ is a $\mathfrak{X}$-refined direct $(\Ops\cup G)$-decomposition of $G$ such
that $\mathcal{K}\mathfrak{X}(G)$ refines $\mathcal{R}\mathfrak{X}(G)$ for some Remak
$(\Ops\cup G)$-decomposition of $G$, then $\mathcal{K}$ is a Remak 
$(\Ops\cup G)$-decomposition of $G$.
\end{lem}
\begin{proof}
As $\mathcal{R}$ is a Remak $(\Ops\cup G)$-decomposition
of $G$, by \lemref{lem:KRS}, $\mathcal{R}\mathfrak{X}(G)$ refines $\mathcal{K}\mathfrak{X}(G)$ 
and so $\mathcal{K}\mathfrak{X}(G)=\mathcal{R}\mathfrak{X}(G)$.  Hence, $|\mathcal{K}-\mathfrak{X}|=|\mathcal{R}-\mathfrak{X}|$ and because $\mathcal{K}$ is $\mathfrak{X}$-refined we also have:
$|\mathcal{K}\cap\mathfrak{X}|=|\mathcal{R}\cap\mathfrak{X}|$.  Therefore,
$|\mathcal{K}|=|\mathcal{K}-\mathfrak{X}|+|\mathcal{K}\cap\mathfrak{X}|
=|\mathcal{R}-\mathfrak{X}|+|\mathcal{R}\cap\mathfrak{X}|=|\mathcal{R}|$.
As every Remak $(\Ops\cup G)$-decomposition of $G$ has the same size, it follows that
$\mathcal{K}$ cannot be refined by a larger direct $(\Ops\cup G)$-decomposition of $G$.
Hence $\mathcal{K}$ is a Remak $(\Ops\cup G)$-decomposition of $G$.
\end{proof}

\begin{thm}\label{thm:merge}
There is a polynomial-time algorithm which, given $G\in\mathbb{G}_n$,
sets $\mathcal{A},\mathcal{H}\subseteq\mathbb{G}_n$, and a function $\theta:\Ops\to\Aut G$,
such that 
\begin{enumerate}[(a)]
\item $\mathcal{A}$ is a Remak $(\Ops\cup G)$-decomposition of $\mathfrak{X}(G)$,
\item $\forall H\in\mathcal{H}$, $\mathfrak{X}(H)=\mathfrak{X}(G)$,
\item $\mathcal{H}/\mathfrak{X}(G)$ is a direct $\Ops$-decomposition of $G/\mathfrak{X}(G)$;
\end{enumerate}
returns an $\mathfrak{X}$-refined direct $\Ops$-decomposition $\mathcal{K}$ of $G$ with the following property. If $\mathcal{R}$ is a direct $\Ops$-decomposition
of $G$ where $\mathcal{H}$ refines $\mathcal{R}\mathfrak{X}(G)$
then $\mathcal{K}\mathfrak{X}(G)$ refines
$\mathcal{R}\mathfrak{X}(G)$; in particular, if $\mathcal{R}$ is Remak then $\mathcal{K}$ is Remak.
\end{thm}
\begin{proof}
\emph{Algorithm.}
\begin{code}{Merge$(~\mathcal{A},~\mathcal{H}~)$}
    $\mathcal{K} = \mathcal{A}$;\\
    $\forall H\in\mathcal{H}$\\
    \begin{block*}
		$\mathcal{K}=${\tt Extend}$(~\langle H,\mathcal{K}\rangle,~\mathcal{K}~)$;
    \end{block*}
    \creturn{$\mathcal{K}$}
\end{code}

\emph{Correctness.} Fix a direct $\Ops$-decomposition $\mathcal{R}$ of $G$
where $\mathcal{H}$ refines $\mathcal{R}\mathfrak{X}(G)$.  We can assume
$\mathcal{R}$ is $\mathfrak{X}$-refined.

The loop runs through a maximal chain $\mathscr{C}$ of subsets of $\mathcal{H}$ and so we
track the iterations by considering the members of $\mathscr{C}$.  By \propref{prop:chain-chain},
$\mathcal{L}=\{L=L_{\mathcal{C}}=\langle\mathcal{C},\mathfrak{X}(G)\rangle: 
\mathcal{C}\in\mathscr{C}\}$ is a direct chain.
We claim the following properties as loop invariants.
At the iteration $\mathcal{C}\in\mathscr{C}$, we claim that $(\mathcal{C},L,\mathcal{K})$ satisfies:
\begin{enumerate}[(P.1)]
\item\label{P:1} $\mathfrak{X}(L)=\mathfrak{X}(G)$,
\item\label{P:3} $\mathcal{K}\mathfrak{X}(G)$ refines
$\mathcal{R}\mathfrak{X}(G)\cap L$, and
\item\label{P:4} $\mathcal{K}$ is an $\mathfrak{X}$-refined direct $(\Ops\cup G)$-decomposition 
of  $L$.
\end{enumerate}
Thus, when the loop completes, $L=\langle\mathcal{H}\rangle=G$.  
By (P.\ref{P:3})
$\mathcal{K}\mathfrak{X}(G)$ refines $\mathcal{R}\mathfrak{X}(G)$.  By (P.\ref{P:4}), $\mathcal{K}$ 
is an $\mathfrak{X}$-refined direct $\Ops$-decomposition of $G$.  
Following \lemref{lem:count}, if $\mathcal{R}$ is a Remak 
$(\Ops\cup G)$-decomposition of $G$ then $\mathcal{K}$ is a Remak 
$(\Ops\cup G)$-decomposition.  We prove (P.\ref{P:1})--(P.\ref{P:4}) by induction.

As we begin with $\mathcal{K}=\mathcal{A}$, in the base case $\mathcal{C}=\emptyset$, 
$L=\mathfrak{X}(G)$, and so (P.\ref{P:1}) holds.
As $\mathcal{K}\mathfrak{X}(G)=\emptyset$ and $\mathcal{R}\mathfrak{X}(G)\cap\mathfrak{X}(G)=\emptyset$
we have (P.\ref{P:3}).  Also (P.\ref{P:4}) holds because of (a).

Now suppose for induction that for some $\mathcal{C}\in\mathscr{C}$, 
$(\mathcal{C},L,\mathcal{K})$ satisfies (P.\ref{P:1})--(P.\ref{P:4}).  Let
$\mathcal{D}=\mathcal{C}\sqcup\{H\}\in\mathscr{C}$ be the successor to $\mathcal{C}$,
for the appropriate $H\in\mathcal{H}-\mathcal{C}$.
Set $M=\langle H,L\rangle$, and 
$\mathcal{M}={\tt Extend}(M,\mathcal{K})$. 
Since $H\leq M$ it follows from (b) that $\mathfrak{X}(G)\leq \mathfrak{X}(M)
\leq \mathfrak{X}(H)=\mathfrak{X}(G)$ so that $\mathfrak{X}(M)=\mathfrak{X}(G)$; hence, (P.\ref{P:1}) holds for $(\mathcal{D}, M,\mathcal{M})$.  

Next we prove (P.\ref{P:3}) holds for $(\mathcal{D}, M,\mathcal{M})$.   
As $L,M\in\mathcal{L}$ and $\mathcal{L}$
is a direct chain with directions $\mathcal{R}$, $\mathcal{R}\cap L$ and $\mathcal{R}\cap M$ 
are direct $(\Ops\cup G)$-decomposition of $L$ and $M$, respectively.
Following \thmref{thm:Extend}(i),
$|\mathcal{M}-\mathcal{K}|\leq 1$.  As $H\nleq L$, $\mathcal{M}\neq \mathcal{K}$,
and there is a group $\lfloor H\rfloor$ in 
$\mathcal{M}-\mathcal{K}$ with $H\leq \lfloor H\rfloor\mathfrak{X}(G)$.
By assumption, $\mathcal{H}$ refines $\mathcal{R}\mathfrak{X}(G)$.  
Hence, there is a unique $R\in \mathcal{R}-\mathfrak{X}$
such that $\mathfrak{X}(G)<H\leq R\mathfrak{X}(G)$.   Indeed, $R$ is the direction 
of $L$. Let 
$C=\langle (\mathcal{R}-\{R\})-\mathfrak{X}\rangle$ and define
	$$\mathcal{J}=\{K\in\mathcal{K}-\mathfrak{X}: K\leq C\mathfrak{X}(G)\}.$$
As the direction of $L$ is $R$, 
$C\mathfrak{X}(G)\cap M=C\mathfrak{X}(G)\cap L=\langle\mathcal{J}\rangle\mathfrak{X}(G)$ 
and by \thmref{thm:chain}, $\mathcal{J}$ lies in a $\mathfrak{X}$-separated direct $(\Ops\cup G)$-decomposition of $M$.
Thus, by \thmref{thm:Extend}(ii), $\mathcal{J}\subseteq \mathcal{M}\cap \mathcal{K}$.
Also, $M = \langle\mathcal{M}-\mathcal{J}\rangle\times \langle \mathcal{J}\rangle$ 
and $\mathfrak{X}(M)=\mathfrak{X}(G)$, so
\begin{align*}
	M/\mathfrak{X}(G)
		& = \langle\mathcal{M}-\mathcal{J}\rangle\mathfrak{X}(G)/\mathfrak{X}(G)\times
				\langle\mathcal{J}\rangle\mathfrak{X}(G)/\mathfrak{X}(G)\\
		& = \langle\mathcal{M}-\mathcal{J}\rangle\mathfrak{X}(G)/\mathfrak{X}(G)\times
				(C\mathfrak{X}(G) \cap M)/\mathfrak{X}(G).
\end{align*}
Thus, $\langle\mathcal{M}-\mathcal{J}\rangle\mathfrak{X}(G)\cap C\mathfrak{X}(G)=\mathfrak{X}(G)$.
Suppose that $X$ is a directly $(\Ops\cup G)$-indecomposable direct $(\Ops\cup G)$-factor of 
$\langle\mathcal{M}-\mathcal{J}\rangle\mathfrak{X}(G)$ which does not lie in $\mathfrak{X}$.
As $\mathcal{R}\cap M$ is a direct $(\Ops\cup M)$-decomposition of $M$ and $X$ lies in a
Remak $(\Ops\cup G)$-decomposition of $M$, 
then by \lemref{lem:KRS}, $X\leq R\mathfrak{X}(M)=R\mathfrak{X}(G)$ or 
$X\leq C\mathfrak{X}(M)=C\mathfrak{X}(G)$.  
Yet, $X\not\in\mathfrak{X}$ so that $X\nleq \mathfrak{X}(G)$ and 
	$$X\cap C\mathfrak{X}(G)\leq 
		\langle\mathcal{M}-\mathcal{J}\rangle\mathfrak{X}(G)\cap C\mathfrak{X}(G)
			=\mathfrak{X}(G);$$
hence, $X\nleq C\mathfrak{X}(G)$.  Thus, $X\leq R\mathfrak{X}(G)$ and as $X$ is 
arbitrary, we get
	$$\langle\mathcal{M}-\mathcal{J}\rangle\mathfrak{X}(G)\leq R\mathfrak{X}(G).$$  
As $M/\mathfrak{X}(G)=(R\mathfrak{X}(G)\cap M)/\mathfrak{X}(G)\times (C\mathfrak{X}(G)\cap M)/\mathfrak{X}(G)$
we indeed have 
	$$\langle\mathcal{M}-\mathcal{J}\rangle\mathfrak{X}(G)= R\mathfrak{X}(G)\cap M.$$
In particular, $\mathcal{M}\mathfrak{X}(G)$ refines $\mathcal{R}\mathfrak{X}(G)\cap M$ and so 
(P.\ref{P:3}) holds.

Finally to prove (P.\ref{P:4}) it suffices to show that $\lfloor H\rfloor$ has no direct $(\Ops\cup G)$-factor in $\mathfrak{X}$.  Suppose otherwise:
so $\lfloor H\rfloor$ has a direct $(\Ops\cup G)$-decomposition $\{H_0,A\}$ where $A\in\mathfrak{X}$ 
and $A$ is directly $(\Ops\cup G)$-indecomposable.
Swap out $\lfloor H\rfloor$ in $\mathcal{M}$ for $\{H_0,A\}$ creating
	$$\mathcal{M}'=(\mathcal{M}-\{\lfloor H\rfloor \})\sqcup\{H_0,A\}
		=(\mathcal{M}\cap\mathcal{K})\sqcup\{H_0,A\}.$$
As $A\in \mathfrak{X}$ it follows that
$A\leq \mathfrak{X}(M)=\mathfrak{X}(G)=\mathfrak{X}(L)$.  In particular, $A\leq L\leq M$.  As 
$A$ is a direct $(\Ops\cup G)$-factor of $M$, $A$ is also a direct 
$(\Ops\cup G)$-factor of $L$.
Since $\langle A,\mathcal{M}\cap \mathcal{K}\rangle\leq L$ it follows that
	$$\mathcal{M}'\cap L=\{H_0 \cap L, A\}\sqcup (\mathcal{M}\cap \mathcal{K})$$ 
is a direct $(\Ops\cup G)$-decomposition of $L$.
Furthermore, $A$ is directly $(\Ops\cup G)$-in\-de\-comp\-o\-sa\-ble, $A\in\mathfrak{X}$, 
and $A$ lies in a Remak $(\Ops\cup G)$-decomposition of $L$.  Also
$\mathcal{K}\cap\mathfrak{X}$ lies in a Remak $(\Ops\cup G)$-decomposition $\mathcal{T}$
of $L$ in which $\mathcal{K}\cap\mathfrak{X}=\mathcal{T}\cap\mathfrak{X}$ 
(\propref{prop:direct-class}(iv) and (v)); thus, by \thmref{thm:KRS} there is a 
$B\in\mathcal{K}\cap \mathfrak{X}$ such that 
	$$(\mathcal{M}'\cap L-\{A\})\sqcup \{B\}$$ 
is a direct $(\Ops\cup G)$-decomposition of $L$.  Hence, 
$\mathcal{M}''=(\mathcal{M}'-\{A\})\sqcup\{B\}$
is a direct $(\Ops\cup G)$-decomposition of $M$.  However, 
$\mathcal{M}''\cap \mathcal{K}=(\mathcal{M}\cap \mathcal{K})\cup\{B\}$.  By 
\thmref{thm:Extend}(i), $\mathcal{M}\cap \mathcal{K}$ is maximal with respect to inclusion in 
$\mathcal{K}$, such 
that $\mathcal{M}\cap \mathcal{K}$ is contained in a direct $(\Ops\cup G)$-decomposition
of $M$.  Thus, $B\in\mathcal{M}\cap\mathcal{K}$.  That is, impossible since it would
imply that $\mathcal{M}'\cap L$ and $(\mathcal{M}'-\{A\})\cap L$ are both direct 
$(\Ops\cup G)$-decompositions of $L$, i.e. that $A\cap L=1$,  But $1<A\leq L$.
This contradiction demonstrates that $\lfloor H\rfloor$ has no direct 
$(\Ops\cup G)$-factor in $\mathfrak{X}$.  Therefore, $\mathcal{M}$ is $\mathfrak{X}$-refined.

Having shown that $M$ and $\mathcal{M}$ satisfy (P.\ref{P:1})--(P.\ref{P:4}), 
at the end of the loop
$\mathcal{K}$ and $L$ are reassigned to $\mathcal{M}$ and $M$ respectively and
so maintain the loop invariants.

\emph{Timing.}  The algorithm loops over every element of $\mathcal{H}$ applying the
polynomial-time algorithm of \thmref{thm:Extend} once in each loop.  Thus, {\tt Merge}
is a polynomial-time algorithm.
\end{proof}

\section{Bilinear maps and $p$-groups of class $2$}\label{sec:bi}
In this section we introduce bilinear maps and a certain commutative ring as a means to access direct decompositions of a $p$-group of class $2$.  In our minds, those groups represent the most difficult case of the direct product problem.  This is because $p$-groups of class $2$ have so many normal subgroups, and many of those pairwise intersect trivially making them appear to be direct factors when they are not.  Thus, a greedy search is almost certain to fail.  Instead, we have had to consider a certain commutative ring that can be derived from a $p$-group.  As commutative rings have unique Remak decomposition, and a decomposable $p$-group will have many Remak decompositions, we might expect such a method to have lost vital information.  However, in view of results such as \thmref{thm:Lift-Extend} we recognize that in fact what we will have constructed leads us to a matching for the extension $1\to \zeta_1(G)\to G\to G/\zeta_1(G)\to 1$.

Unless specified otherwise, in this section $G$ is a $p$-group of class $2$.

\subsection{Bilinear maps}
Here we introduce $\Ops$-bilinear maps and direct 
$\Ops$-decompositions of $\Ops$-bilinear maps.  This allows us to solve the match 
problem for $p$-groups of class $2$.

Let $V$ and $W$ denote abelian $\Ops$-groups.
A map $b:V\times V\to W$ is $\Ops$-\emph{bilinear} if
\begin{align}
	b(u+u',v+v') & = b(u,v)+b(u',v)+b(u,v')+b(u',v'), \textnormal{ and }\\
        b(ur,v) & = b(u,v)r  = b( u,vr),
\end{align}
for all $u,u'v,v'\in V$ and all $r\in \Ops$.  Every $\Ops$-bilinear map is also 
$\mathbb{Z}$-bilinear.  Define
\begin{equation}
	b(X,Y) = \langle b(x,y) : x\in X, y\in Y\rangle
\end{equation}
for $X,Y\subseteq V$.  If $X\leq V$ then define the \emph{submap}
\begin{equation}
	b_X:X\times X\to b(X,X)
\end{equation}
as the restriction of $b$ to inputs from $X$.
The \emph{radical} of $b$ is
\begin{equation}
	\rad b = \{ v\in V : b(v,V)=0=b(V,v) \}.
\end{equation}
We say that $b$ is \emph{nondegenerate} if $\rad b=0$.
Finally, call $b$ \emph{faithful} $\Ops$-bilinear when 
$(0:_{\Ops} V)\cap (0:_{\Ops} W)=0$,
where $(0:_{\Ops} X)=\{r\in \Ops: Xr=0\}$, $X\in \{V,W\}$.

\begin{defn}
Let $\mathcal{B}$ be a family of $\Ops$-bilinear maps 
$b:V_b\times V_b\to W_b$, $b\in\mathcal{B}$.
Define $\oplus\mathcal{B}=\bigoplus_{b\in\mathcal{B}} b$ as the 
$\Ops$-bilinear map 
$\bigoplus_{b\in\mathcal{B}} V_b\times\bigoplus_{b\in\mathcal{B}} V_b
\to \bigoplus_{b\in\mathcal{B}} W_b$ where:
\begin{equation}
	\left(\oplus\mathcal{B}\right)\left(
		(u_b)_{b\in\mathcal{B}},(v_b)_{b\in\mathcal{B}}\right)
			= (b(u_b,v_b))_{b\in\mathcal{B}},\qquad
			\forall (u_b)_{b\in\mathcal{B}},(v_b)_{b\in\mathcal{B}}
				\in \bigoplus_{b\in\mathcal{B}}V_b.
\end{equation} 
\end{defn}

\begin{lem}\label{lem:internal-direct-bi}
If $b:V\times V\to W$ is an $\Ops$-bilinear map, 
$\mathcal{C}$ a finite set of submaps of $b$ such that 
\begin{enumerate}[(i)]
\item $\{X_c: c:X_c\times X_c\to Z_c\in\mathcal{C}\}$ is a direct 
$\Ops$-decomposition of $V$, 
\item $\{Z_c: c:X_c\times X_c\to Z_c\in\mathcal{C}\}$ is a direct 
$\Ops$-decomposition of $W$, and
\item $b(X_c,X_{d})=0$ for distinct $c,d\in\mathcal{C}$;
\end{enumerate}
then
$b=\bigoplus\mathcal{C}$.
\end{lem}
\begin{proof}
By $(i)$, we may write each $u\in V$ as 
$u=(u_c)_{c\in\mathcal{C}}$ with $u_c\in X_c$, for all
$c:X_c\times X_c\to Z_c\in
\mathcal{C}$.  By $(iii)$ followed by $(ii)$ we have that
$b(u,v)=\sum_{c,d\in\mathcal{C}} b(u_c,v_d)
		=\sum_{c\in\mathcal{C}} c(u_c,v_c)
		=\left(\oplus\mathcal{C}\right)(u,v).$
\end{proof}

\begin{defn}\label{def:ddecomp-bi}
A \emph{direct $\Ops$-decomposition} of an $\Ops$-bilinear map $b:V\times V\to W$
is a set $\mathcal{B}$ of submaps of $b$ 
satisfying the hypothesis of \lemref{lem:internal-direct-bi}.
Call $b$ directly $\Ops$-indecomposable if its only direct $\Ops$-decomposition
is $\{b\}$.
A Remak $\Ops$-decomposition of $b$ is an $\Ops$-decompositions
whose members or directly $\Ops$-indecomposable.
\end{defn}

The bilinear maps we consider were created by Baer \cite{Baer:class-2} and are the
foundation for the many Lie methods that have been associated to $p$-groups.  Further
details of our account can be found in \cite[Section 5]{Warfield:nil}.

The principle example of such maps is the commutation of an $\Ops$-group $G$ where
$\gamma_2(G)\leq \zeta_1(G)$.  There we define $V=G/\zeta_1(G)$, $W=\gamma_2(G)$, and
$b=\Bi(G):V\times V\to W$ where
\begin{equation}\label{eq:Bi}
	b(\zeta_1(G) x,\zeta_1(G) y )=b(x,y),\qquad \forall x,\forall y, x,y\in G.
\end{equation}
It is directly verified that $b$ is $\mathbb{Z}_{p^e}[\Ops]$-bilinear where $G^{p^e}=1$, and
furthermore, nondegenerate.  When working in $V$ and $W$ we use additive notation.

Given $H\leq G$ we define $U=H\zeta_1(G)/\zeta_1(G)\leq V$,
$Z=H\cap \gamma_2(G)\leq W$, and $c:=\Bi(H;G):U\times U\to Z$ where
\begin{equation}
	c(u,v) = b(u,v),\qquad \forall u\forall v, u,v\in U.
\end{equation}

\begin{prop}\label{prop:p-group-bi}
If $G$ is a $\Ops$-group and $\gamma_2(G)\leq \zeta_1(G)$, then every direct $\Ops$-decomposition
$\mathcal{H}$ of $G$ induces a direct $\Ops$-decomposition 
\begin{equation}
	\Bi(\mathcal{H}) = \{ \Bi(H; G) : H\in\mathcal{H}\}.
\end{equation}
If $\Bi(P)$ is directly $\Ops$-indecomposable and $\zeta_1(G)\leq \Phi(G)$, then
$G$ is directly $\Ops$-indecomposable.
\end{prop}
\begin{proof}
Set $b:=\Bi(G)$.
By \lemref{lem:induced} and \propref{prop:V-inter-1}, $\mathcal{H}\zeta_1(G)/\zeta_1(G)$
is a direct $\Ops$-decomposition of $V=G/\zeta_1(G)$ and $\mathcal{H}\cap \gamma_2(G)$ is
a direct $\Ops$-decomposition of $W=\gamma_2(G)$.  Furthermore, for each $H\in\mathcal{H}$,
$$b(H\zeta_1(G)/\zeta_1(G), \langle\mathcal{H}-\{H\}\rangle\zeta_1(G)/\zeta_1(G))
=[H,\langle\mathcal{H}-\{H\}\rangle]=0\in W.$$  
In particular, $\Bi(\mathcal{H})$ is
a direct $\Ops$-decomposition of $b$.

Finally, if $\Bi(P)$ is directly indecomposable then $|\Bi(\mathcal{H})|=1$.  Thus,
$\mathcal{H}\zeta_1(G)=\{G\}$.  Therefore $\mathcal{H}$ has exactly one non-abelian member.
Take $Z\in\mathcal{H}\cap\mathfrak{A}$.  As $Z$ is abelian, $Z\leq \zeta_1(G)$.  If
$\zeta_1(G)\leq \Phi(G)$ then the elements of $G$ are non-generators.  In particular,
$G=\langle\mathcal{H}\rangle=\langle \mathcal{H}-\{Z\}\rangle$.  But by definition no
proper subset of decomposition generates the group.  So $\mathcal{H}\cap\mathfrak{A}=\emptyset$.
Thus, $\mathcal{H}=\{G\}$ and $G$ is directly $\Ops$-indecomposable.
\end{proof}

Baer and later others observed a partial reversal of the map $G\mapsto \Bi(G)$.  Our account
follows \cite{Warfield:nil}.  In particular,
if $b:V\times V\to W$ is a $\mathbb{Z}_{p^e}$-bilinear map then we may define a group
$\Grp(b)$ on the set $V\times W$ where the product is given by:
\begin{equation}
	(u,w)*(u',w') = (u+u', w+b(u,u')+w'),
\end{equation}
for all $(u,w)$ and all $(u',w')$ in $V\times W$.  The following are immediate from the definition.
\begin{enumerate}[(i)]
\item $(0,0)$ is the identity and for all $(u,w)\in V\times W$, $(u,w)^{-1}=(-u,-w + b(u,u))$.
\item For all $(u,w)$ and all $(v,w)$ in $V\times W$, $[(u,w), (v,w')] = (0, b(u,v)-b(v,u))$.
\end{enumerate}
If $b$ is $\Ops$-bilinear then $\Grp(b)$ is an $\Ops$-group where 
$$\forall s\in \Ops, \forall (u,w)\in V\times W,\qquad (u,w)^{s}=(u^s, w^s).$$
In light of (ii), if $p>2$ and $b$ is alternating, i.e. for all $u$ and all $v$ in $V$, 
$b(u,v)=-b(v,u)$, then $[(u,w),(v,w')]=(0,2b(u,v))$.  For that reason it is typical to consider
$\Grp(\frac{1}{2}b)$ in those settings so that $[(u,w),(v,w')]=(0,b(u,v))$.  We shall not require
this approach.  If $G^p=1$ then $G\cong \Grp(\Bi(G))$ \cite[Proposition 3.10(ii)]{Wilson:unique-cent}.

\begin{coro}\label{coro:exp-p}
If $G$ is a $p$-group with $G^p=1$ and $\gamma_2(G)\leq \zeta_1(G)$ then $G$ is directly
$\Ops$-indecomposable if, and only if, $\Bi(G)$ is directly $\Ops$-indecomposable and 
$\zeta_1(G)\leq \Phi(G)$.
\end{coro}
\begin{proof}
The reverse directions is \propref{prop:p-group-bi}.  We focus on the forward direction.
As $G^p=1$ it follows that $G\cong \Grp(\Bi(G))=:\hat{G}$.  Set 
$b:=\Bi(G)$.  Let $\mathcal{B}$ be a direct $\Ops$-decomposition 
of $b$, and therefore also of $\Bi(G)$.  For each $c:X_c\times X_c\to Z_c\in\mathcal{B}$, 
define $\Grp(c,b)=X_c\times Z_c\leq V\times W$.  We claim that $\Grp(c;b)$ is an $\Ops$-subgroup of
$\Grp(b)$.  In particular, $(0,0)\in \Grp(c;b)$ and for all $(x,w),(y,w')\in \Grp(c;b)$,
$(x,w)*(-y,-w'+b(y,y))= (x-y, w-b(x,y)-w'+b(y,y))\in X_c\times Z_c=\Grp(c;b)$.  Furthermore,
\begin{align*}
	\left[\Grp(c;b), \Grp\left( \sum_{d\in\mathcal{C}-\{c\}} d; b\right) \right]
		& = \left( 0, 2b\left( X_c, \sum_{d\in\mathcal{C}-\{c\}} X_d\right) \right)=(0,0).
\end{align*}
Combined with $\Grp(b)=\langle \Grp(c;b): c\in\mathcal{C}\rangle$ it follows that 
$\Grp(c;b)$ is normal in $\Grp(b)$.  Finally, 
\begin{align*}
\Grp(c;b) \cap \Grp\left(\sum_{d\in\mathcal{C}-\{c\}} d; b\right)
	& = (X_c\times Z_c) \cap \sum_{d\in\mathcal{C}-\{c\}}(X_d\times Z_d) = 0\times 0.
\end{align*}
Thus, $\mathcal{H}=\{\Grp(c;b): c\in\mathcal{C}\}$ is a direct $\Ops$-decomposition of $\Grp(b)$.
As $G$ is directly $\Ops$-indecomposable it follows that $\mathcal{H}=\{G\}$ and so
$\mathcal{C}=\{b\}$.  Thus, $b$ is directly $\Ops$-indecomposable.
\end{proof}

\subsection{Centroids of bilinear maps}
\label{sec:enrich}
In this section we replicate the classic interplay of idempotents of a ring and direct decompositions
of an algebraic object, but now for context of bilinear maps.  The relevant ring is the centroid, 
defined similar to centroid of a nonassociative ring \cite[Section X.1]{Jacobson:Lie}.  As with 
nonassociative rings, the idempotents of the centroid of a bilinear map correspond to direct decompositions.   
Myasnikov \cite{Myasnikov} may have been the 
first to generalize such methods to bilinear maps.

\begin{defn}\label{def:centroid}
The \emph{centroid} of an $\Ops$-bilinear $b:V\times V\to W$ is
\begin{equation*}
	C_{\Ops}(b) = \{ (f,g)\in \End_{\Ops} V\oplus \End_{\Ops} W:
		b(uf,v)=b(u,v)g=b(u,vf),\forall u,v\in V\}.
\end{equation*}
If $\Ops=\emptyset$ then write $C(b)$.
\end{defn}

\begin{lem}\label{lem:centroid}
Let $b:V\times V\to W$ be an $\Ops$-bilinear map.  Then the following hold.
\begin{enumerate}[(i)]
\item $C_{\Ops}(b)$ is a subring of $\End_{\Ops} V\oplus \End_{\Ops} W$, and 
$V$ and $W$ are $C_{\Ops}(b)$-modules.  

\item If $b$ is $K$-bilinear for a ring $K$, then $K/(0:_K V)\cap (0:_K W)$ embeds
in $C(b)$.  
In particular, $C(b)$ is the largest 
ring over which $b$ is faithful bilinear.

\item If $b$ is nondegenerate and $W=b(V,V)$ then $C_{\Ops}(b)=C(b)$
and $C(b)$ is commutative.
\end{enumerate}
\end{lem}
\begin{proof}
Parts (i) and (ii) are immediate from the definitions.  For part (iii), if $s\in \Ops$ and
$(f,g)\in C(b)$, then $b((su)f,v)=b(su,vf)=sb(u,vf)=b(s(uf),v)$ for all $u$ and
all $v\in V$.  As $b$ is nondegenerate and
$b((su)f-s(uf),V)=0$, it follows that $(su)f=s(uf)$.  In a similar fashion, $g\in\End_{\Ops}W$
so that $(f,g)\in C_{\Ops}(b)$.  For part (iii) we repeat the same shuffling game
above: if $(f,g),(f',g')\in C(b)$ then $b(u(ff'),v)=b(u,vf)f'=b(u(f'f),v)$.  By the 
nondegenerate assumption we get that $ff'=f'f$ and also $gg'=g'g$.
\end{proof}

\begin{remark}
If $\rad b=0$ and $(f,g),(f',g)\in C(b)$ then $f=f'$.  If
$W=b(V,V)$ and $(f,g),(f,g')\in C(b)$ then $g=g'$.  So if $\rad b=0$
and $W=b(V,V)$ then the first variable determines the second and vice-versa.
\end{remark}

\subsection{Idempotents, frames, and direct $\Ops$-decompositions}\label{sec:bi-direct}
In this section we extend the usual interplay of idempotents and direct decompositions to the context of bilinear maps and them $p$-groups of class $2$.  This allows us to prove \thmref{thm:indecomp-class2}.
This section follows the notation described in Subsection \ref{sec:rings}.

\begin{lem}\label{lem:idemp}
Let $b:V\times V\to W$ be an $\Ops$-bilinear map.  
\begin{enumerate}[(i)]
\item A set $\mathcal{B}$ of $\Ops$-submaps of $b$ is a direct $\Ops$-decomposition 
of $b$ if, and only if,
\begin{equation}
	\mathcal{E}(\mathcal{B}) 
		=\{ (e(V_c),e(W_c)) : c:V_c\times V_c\to W_c\in \mathcal{B}\}.
\end{equation}
is a set of pairwise orthogonal idempotents of $C_{\Ops}(b)$ which sum to $1$.
\item
$\mathcal{B}$ is a Remak $\Ops$-decomposition of $b$ if, and only if,
$\mathcal{E}(\mathcal{B})$ is a frame.  
\item If $b$ is nondegenerate and $W=b(V,V)$,
then $b$ has a unique Remak $\Ops$-decomposition of $b$.
\end{enumerate}
\end{lem}
\begin{proof}
For $(i)$, by \defref{def:ddecomp-bi}, $\{V_b: b\in\mathcal{B}\}$ is a direct
decomposition of $V$ and $\{W_b:b\in\mathcal{B}\}$ is a direct decomposition 
of $W$.  Thus, $\mathcal{E}(\mathcal{B})$ is a set of
pairwise orthogonal idempotents which sum to $1$.

Let $(e,f)\in\mathcal{E}(\mathcal{X})$. As 
$1-e=\sum_{(e',f')\in\mathcal{E}(\mathcal{B})-\{(e,f)\}} e'$ it follows that 
for all $u,v\in V$ we have $b(ue,v(1-e))\in b(Ve,V(1-e))=0$ by 
the assumptions on $\mathcal{B}$.  Also, $b(ue,v e)\in Wf$. Together we have:
\begin{eqnarray*}
	b(ue,v) & = &  b(ue,v e) + b(ue,v(1-e))
		= b(ue,v e),\\
	b(u,ve) & = &  b(u(1-e),v e) + b(ue,ve)
		= b(ue,v e),\textnormal{ and }\\
	b(u,v)f & = & \left(\sum_{(e',f')\in \mathcal{E}(\mathcal{B})}
		b(ue',ve')f'\right)f	= b(ue,v e)f=b(ue,ve).
\end{eqnarray*}
Thus $b(ue,v)=b(u,v)f=b(u,ve)$ which proves
$(e,f)\in C_{\Ops}(b)$; hence, $\mathcal{E}(\mathcal{B})\subseteq C_{\Ops}(b)$.

Now suppose that $\mathcal{E}$ is a set of pairwise orthogonal idempotents 
of $C_{\Ops}(b)$ which sum to $1$.  It follows that
$\{Ve: (e,f)\in\mathcal{E}\}$ is a direct $\Ops$-decomposition of $V$ and
$\{Wf: (e,f)\in\mathcal{E}\}$ is a direct $\Ops$-decomposition of $W$.  Finally,
$b(ue,ve')=b(uee',v)=0$.  Thus, $\{b|_{(e,f)}:V_e\times V_e\to W_f: (e,f)\in\mathcal{E}\}$
is a direct $\Ops$-decomposition of $C(b)$.

Now $(ii)$ follows.  For $(iii)$, we now by \lemref{lem:centroid}(ii) that
$C(b)=C_{\Ops}(b)$ is commutative Artinian.  The rest follows from \lemref{lem:lift-idemp}(iv).
\end{proof}

\begin{thm}\label{thm:Match-class2}
If $G$ is a $p$-group and $\gamma_2(G)\leq \zeta_1(G)$, then there is a unique frame
$\mathcal{E}$ in $C(\Bi(G))$.  Furthermore, if $\gamma_2(G)=\zeta_1(G)$ then every Remak $\Ops$-decomposition $\mathcal{H}$ of
$G$ matches a unique partition of $(\mathcal{K},\mathcal{Q})$ where
\begin{align*}
	\mathcal{K} := \{ W\hat{e}: (e,\hat{e})\in \mathcal{E}\},\\
	\mathcal{Q} := \{ Ve : (e,\hat{e})\in\mathcal{E}\}.
\end{align*}
If $G^p=1$ then every Remak $\Ops$-decomposition of $G$ matches $(\mathcal{K},\mathcal{Q})$.
\end{thm}
\begin{proof}
This follows from \propref{prop:p-group-bi}, \lemref{lem:idemp}, and \corref{coro:exp-p}.
\end{proof}

\subsection{Proof of \thmref{thm:indecomp-class2}}
This follows from \thmref{thm:Match-class2}.
\hfill $\Box$

\subsection{Centerless groups}
We close this section with a brief consideration of groups with a trivial center.

\begin{lemma}\label{lem:centerless}
Let $G$ be an $\Ops$-group with $\zeta_1(G)=1$ and $N$ a minimal $(\Ops\union G)$-subgroup of $G$.  Then the following hold.
\begin{enumerate}[(i)]
\item $G$ has a unique Remak $\Ops$-decomposition $\mathcal{R}$.
\item There is a unique $R\in\mathcal{R}$ such that $N\leq R$.
\item $\{C_R(N),\langle\mathcal{R}-\{R\}\rangle\}$ is a direct
$(\Ops\union G)$-decomposition of $C_G(N)$.
\item Every Remak $(\Ops\union G)$-decomposition $\mathcal{H}$ of
$C_G(N)$ refines $\{C_R(N),\langle\mathcal{R}-\{R\}\rangle\}$.
\end{enumerate}
\end{lemma}
\begin{proof}
Given Remak $\Ops$-decompositions $\mathcal{R}$ and $\mathcal{S}$ of $G$, by
\lemref{lem:KRS} and the assumption that $\zeta_1(G)=1$, it follows that 
$\mathcal{R}=\mathcal{R}\zeta_1(G)=\mathcal{S}\zeta_1(G)=\mathcal{S}$.  This proves (i).

For (ii), if $N$ is a minimal $(\Ops\union G)$-subgroup of $G$ then 
$[R,N]\leq R\intersect N\in \{1,N\}$, for all $R\in\mathcal{R}$.  If $[R,N]=1$
for all $R\in\mathcal{R}$ then $N\leq \zeta_1(G)=1$ which contradicts the assumption
that $N$ is minimal.  Thus, for some $R\in\mathcal{R}$, $N\leq R$.  The
uniqueness follows as $R\intersect\langle\mathcal{R}-\{R\}\rangle=1$.

By (ii), $[N,\langle R-\{R\}\rangle]=[R,\langle \mathcal{R}-\{R\}\rangle]=1$ which shows 
$\langle \mathcal{R}-\{R\}\rangle\leq C_G(N)$.  Hence,
$C_G(N)=C_R(N)\times \langle\mathcal{R}-\{R\}\rangle$.  
This proves (iii).

Finally we prove (iv).  Let $\mathcal{K}$ be a Remak $(\Ops\cup G)$-decomposition of $C_G(N)$.  
Let $\mathcal{S}$ be a Remak $(\Ops\cup G)$-decomposition of $C_G(N)$ which refines the direct 
$(\Ops\cup G)$-decomposition $C_G(N)=C_R(N)\times \langle\mathcal{R}-\{R\}\rangle$
given by (iii).  Note that $\mathcal{R}-\{R\}\subseteq \mathcal{S}$ as
members of $\mathcal{R}$ cannot be refined further.  By
\thmref{thm:KRS}, there is a $\mathcal{J}\subseteq\mathcal{K}$ such that we may exchange 
$\mathcal{R}-\{R\}\subseteq \mathcal{S}$ with $\mathcal{J}$; hence, $\{C_R(N)\}\sqcup \mathcal{J}$ 
is a direct $(\Ops\union G)$-decomposition of $C_G(N)$.  Now $R\intersect \langle\mathcal{J}\rangle
\leq C_R(N)\intersect \langle\mathcal{J}\rangle=1$.  Also
\begin{equation} 
	\langle R,\mathcal{J}\rangle=\langle R, C_R(N),\mathcal{J}\rangle
		=\langle R,\mathcal{R}-\{R\}\rangle=G.
\end{equation}
As every member of $\mathcal{J}$ is an $(\Ops\union G)$-subgroup of $G$, it follows 
that the are normal in $G$ and so $\{R\}\sqcup\mathcal{J}$ is a direct $\Ops$-decomposition 
of $G$.  As the members of $\mathcal{J}$ are $\Ops$-indecomposable it follows that 
$\{R\}\sqcup\mathcal{J}$ is a Remak $\Ops$-decomposition of $G$.  However, $G$ has a unique 
Remak $\Ops$-decomposition so $\mathcal{J}=\mathcal{R}-\{R\}$.  As $\mathcal{J}$ was a subset 
of an arbitrary Remak $(\Ops\union G)$-decomposition of $C_G(N)$ it follows that every Remak 
$(\Ops\union G)$-decomposition of $C_G(N)$ contains $\mathcal{R}-\{R\}$. 
\end{proof}

\begin{prop}\label{prop:centerless-Extend}
For groups $G$ with $\zeta_1(G)=1$, the set $\mathcal{M}$ of minimal 
$(\Ops\cup G)$-subgroups is a direct $(\Ops\cup G)$-decomposition of the socle of $G$
and furthermore there is a unique partition of $\mathcal{M}$ which extends to the
Remak $\Ops$-decomposition of $G$.
\end{prop}

The following consequence shows how the global Remak decomposition of a group with
trivial solvable radical is determined precisely from a unique partition of the Remak
decomposition of it socle.   
\begin{coro}
If $G$ has trivial solvable radical and $\mathcal{R}$ is its Remak decomposition then
$\mathcal{R}=\{C_G(C_G(\soc(R))): R\in\mathcal{R}\}.$
\end{coro}

\section{The Remak Decomposition Algorithm}\label{sec:Remak}
In this section we prove \thmref{thm:FindRemak}.    The approach is to break up a given group
into sections for which a Remak $(\Ops\cup G)$-decomposition can be computed directly.
The base cases include $\Ops$-modules (\corref{coro:FindRemak-abelian}), $p$-groups of class $2$ (which
follows from \thmref{thm:Match-class2}), and groups with a trivial center.  We use \thmref{thm:Lift-Extend} as justification that we can interlace these base cases
to sequentially lift direct decomposition via the algorithm {\tt Merge} of \thmref{thm:merge}.

\subsection{Finding Remak $\Ops$-decompositions for nilpotent groups of class $2$}
\label{sec:FindRemak-class2}
In this section we prove \thmref{thm:FindRemak} for the case of nilpotent groups $G$ 
of class $2$.  The algorithm depends on \thmref{thm:Match-class2} and \thmref{thm:merge}.

To specify a $\mathbb{Z}$-bilinear map $b:V\times V\to W$ for computation we need only provide
the \emph{structure constants} with respect to fixed bases of $V$ and $W$.
Specifically let $\mathcal{X}$ be a basis of $V$ and $\mathcal{Y}$ a basis of $W$.  Define
$B_{xy}^{(z)}\in\mathbb{Z}$ so that the following equation is satisfied:
\begin{align*}
	b\left(\sum_{x\in\mathcal{X}} \alpha_x x,\sum_{y\in\mathcal{X}} \beta_y y \right)
		& = \sum_{z\in\mathcal{Z}} \left(\sum_{x,y\in\mathcal{X}} \alpha_x B_{xy}^{(z)} \beta_y\right)z
			& (\forall x\in\mathcal{X},\forall\alpha_x,\beta_x\in\mathbb{Z}).
\end{align*}

\begin{lem}\label{lem:Remak-bilinear}
There is a deterministic polynomial-time algorithm, which given $\Ops$-modules
$V$ and $W$ and a nondegenerate $\Ops$-bilinear map $b:V\times V\to W$
with $W=b(V,V)$, returns a Remak $\Ops$-decomposition of $b$.
\end{lem}
\begin{proof}
\emph{Algorithm}.
Solve a system of linear equations in the (additive) abelian group 
$\End_{\Ops} V\times \End_{\Ops} W$ to find generators for $C_{\Ops}(b)$.
Use {\sc Frame} to find a frame $\mathcal{E}$ 
of $C_{\Ops}(b)$.  Return
$\{b|_{(e,f)}:Ve\times Ve\to Wf : (e,f)\in\mathcal{E}\}$.

\emph{Correctness}.
This is supported by \lemref{lem:idemp} and  \thmref{thm:Frame}.

\emph{Timing}.
This follows from the timing of {\sc Solve} and {\sc Frame}.
\end{proof}

\begin{thm}\label{thm:FindRemak-class2}
There is a polynomial-time algorithm which, given a nilpotent $\Ops$-group of class $2$, returns a Remak $\Ops$-decomposition 
of the group.
\end{thm}
\begin{proof}
Let $G\in\mathbb{G}_n^{\Ops}$ with $\gamma_2(G)\leq \zeta_1(G)$.

\emph{Algorithm}.
Use {\sc Order} to compute $|G|$.  For each prime $p$ dividing $|G|$, write
$|G|=p^e m$ where $(p,m)=1$ and set $P:=G^{m}$.  Set $b_p:=\Bi(P)$.  
Use the algorithm of \lemref{lem:Remak-bilinear} to
find a direct $\Ops$-decomposition $\mathcal{B}$ of $b$.
Define each of the following:
\begin{align*}
	\mathcal{X}(\mathcal{B}) & = \{ X_c : c:X_c\times X_c\to Z_c\in \mathcal{B}\}\\
	\mathcal{H} & = \{H\leq P: \zeta_1(P)\leq H, H/\zeta_1(P)\in \mathcal{X}(\mathcal{B})\}.
\end{align*}
Use \corref{coro:FindRemak-abelian} to build a Remak $\Ops$-decomposition
$\mathcal{Z}$ of $\zeta_1(P)$.  Set $\mathcal{R}_p:={\tt Merge}(\mathcal{Z},\mathcal{H})$.
Return $\bigcup_{p\mid |G|} \mathcal{R}_p$.

\emph{Correctness}.
By \lemref{lem:Remak-bilinear} the set $\mathcal{B}$ is the unique Remak
$\Ops$-decomposition of $\Bi(G)$.  By \thmref{thm:Match-class2} and
\thmref{thm:merge} the return a Remak $\Ops$-decomposition of $G$.

\emph{Timing}.
The algorithm uses a constant number of polynomial time subroutines.
\end{proof}

We have need of one final observation which allows us to modify certain decompositions into ones that match the hypothesis \thmref{thm:merge}(b) when the up grading pair is $(\mathfrak{N}_c,G\mapsto \zeta_c(G))$.
\begin{lemma}\label{lem:centralize}
There is a polynomial-time algorithm which, given an $\Ops$-decomposition $\mathcal{H}=
\mathcal{H}\zeta_c(G)$
of a group $G$, returns the finest $\Ops$-decomposition $\mathcal{K}$ refined by $\mathcal{H}$
and such that for all $K\in\mathcal{K}$, $\zeta_c(K)=\zeta_c(G)$.  (The proof also
shows there is a unique such $\mathcal{K}$.)
\end{lemma}
\begin{proof}
Observe that $\mathcal{K}=\{\langle H\in\mathcal{H}: [K,H,\dots,H]\neq 1\rangle: K\in\mathcal{K}\}$.
We can create $\mathcal{K}$ by a transitive closure algorithm.
\end{proof}

\begin{thm}\label{thm:FindRemak-Q}
{\sc Find-$\Ops$-Remak} has polynomial-time solution.
\end{thm}
\begin{proof}
Let $G\in \mathbb{G}_n^{\Ops}$.

\emph{Algorithm.}
If $G=1$ then return $\emptyset$.  Otherwise, compute $\zeta_1(G)$.  
If $G=\zeta_1(G)$ then use {\sc Abelian.Remak-$\Ops$-Decomposition} and return 
the result.  Else, if $\zeta_1(G)=1$ then use {\sc Minimal-$\Ops$-Normal} to find 
a minimal $(\Ops\cup G)$-subgroup $N$ of $G$.  Use {\sc Normal-Centralizer} to compute 
$C_G(N)$.  If $C_G(N)=1$ then return $\{G\}$.  Otherwise, recurse with 
$C_G(N)$ in the role of $G$ to find a Remak $\Ops$-decomposition 
$\mathcal{K}$ of $C_G(N)$.  Call $\mathcal{H}:=\textalgo{Extend}(G,\mathcal{K})$ to create 
a direct $\Ops$-decomposition $\mathcal{H}$ extending $\mathcal{K}$ maximally.  Return $\mathcal{H}$.

Now $G>\zeta_1(G)>1$.  Compute $\zeta_2(G)$ and use \thmref{thm:FindRemak-class2} to 
construct a Remak $(\Ops\cup G)$-decomposition $\mathcal{A}$ of $\zeta_2(G)$.  If
$G=\zeta_2(G)$ then return $\mathcal{A}$; otherwise, $G>\zeta_2(G)$ (consider \figref{fig:rel-ext}).  
Use a recursive call on $G/\zeta_1(G)$ to find $\mathcal{H}=\mathcal{H}\zeta_1(G)$ such that 
$\mathcal{H}/\zeta_1(G)$ is a Remak $\Ops$-decomposition of $G/\zeta_1(G)$.  Apply \lemref{lem:centralize} to $\mathcal{H}$ and then set $\mathcal{J}:=\textalgo{Merge}(\mathcal{A},\mathcal{H})$, and return $\mathcal{J}$.

\emph{Correctness.}
The case $G=\zeta_1(G)$ is proved by \corref{coro:FindRemak-abelian} and the case $G=\zeta_2(G)$ is proved in \thmref{thm:FindRemak-class2}.  

Now suppose that $G>\zeta_1(G)=1$.  Following \lemref{lem:centerless}, $G$ has a unique Remak $\Ops$-decomposition $\mathcal{R}$ and there is a unique $R\in\mathcal{R}$ such that $N\leq R$ and $\langle\mathcal{R}-\{R\}\rangle\leq C_G(N)$.  So if $C_G(N)=1$ then $G$ is directly indecomposable and the return of the algorithm is correct.  Otherwise the algorithm makes a recursive call to find a Remak $(\Ops\cup G)$-decomposition $\mathcal{K}$ of $C_G(N)$.  By \lemref{lem:centerless}(iv), $\mathcal{K}$ contains $\mathcal{R}-\{R\}$ and so there is a unique maximal extension of $\mathcal{K}$, namely $\mathcal{R}$, and so by \thmref{thm:Extend}, the algorithm $\textalgo{Extend}$ creates the Remak $\Ops$-decomposition of $G$ so the return in this case is correct.

Finally suppose that $G>\zeta_2(G)\geq\zeta_1(G)>1$.  There we have the commutative
diagram \figref{fig:rel-ext} which is exact in rows and columns.
\begin{figure}
\begin{equation*}
\xymatrix{
		& 				  &   1	     &	1 & \\
1\ar[r] & \zeta_2(G)\ar[r] & G \ar[r]\ar[u] & G/\zeta_2(G)\ar[r]\ar[u] & 1\\
1\ar[r] & \zeta_1(G)\ar[r]\ar[u] & G \ar[r]\ar@{=}[u] & G/\zeta_1(G)\ar[r]\ar[u] & 1\\
		& 	1\ar[u]			  &   1\ar[u]	     &	 & \\
}
\end{equation*}
\caption{The relative extension $1<\zeta_1(G)\leq\zeta_2(G)<G$.  The rows and columns are exact.}\label{fig:rel-ext}
\end{figure}
By \thmref{thm:Lift-Extend}, $\mathcal{H}\zeta_2(G)$ refines $\mathcal{R}\zeta_2(G)$ and so
the algorithm \textalgo{Merge} is guaranteed by
\thmref{thm:merge} to return a Remak $\Ops$-decomposition of $G$ (consider \figref{fig:recurse}).
\begin{figure}
\begin{equation}
\xymatrix{
		& 				  &   1	     &	1 & \\
1\ar[r] & \prod\mathcal{A}\ar[r] & \prod{\tt Merge}(\mathcal{A},\mathcal{H}) \ar[r]\ar[u] & \prod \mathcal{H}/\zeta_2(G)\ar[r]\ar[u] & 1\\
1\ar[r] & \zeta_1(G)\ar[r]\ar[u] & G \ar[r]\ar@{=}[u] & \prod\mathcal{H}/\zeta_1(G)\ar[r]\ar[u] & 1\\
		& 	1\ar[u]			  &   1\ar[u]	     &	 & \\
}
\end{equation}
\caption{The recursive step parameters feed into {\tt Merge} to produce a Remak
$\Ops$-decomposition of $G$.}\label{fig:recurse}
\end{figure}

\emph{Timing.}  The algorithm enters a recursive call only if $\zeta_1(G)=1$ or 
$G>\zeta_2(G)\geq \zeta_1(G)>1$.  As these two case are exclusive there is at most one recurse call 
made by the algorithm.  
The remainder of the algorithm uses polynomial time methods as indicated.
\end{proof}

\subsection{Proof of \thmref{thm:FindRemak}}
This is a corollary to \thmref{thm:FindRemak-Q}
\hfill $\Box$

\begin{coro}\label{coro:FindRemak-matrix}
$\textalgo{FindRemak}$ has a deterministic polynomial-time solution for matrix $\Gamma_d$-groups.
\end{coro}
\begin{proof}
This follows from Section \ref{sec:tools}, \remref{rem:matrix}, and \thmref{thm:FindRemak-Q}.
\end{proof}
\subsection{General operator groups}\label{sec:gen-ops}
Now we suppose that $G\in \mathbb{G}_n$ is a $\Ops$-group for a general set $\Omega$ of
operators.  That is, $\Ops\theta\subseteq \End G$.  To solve \textalgo{Remak-$\Ops$-Decomposition}
in full generality it suffices to reduce to the case where $\Omega$ acts as automorphisms on $G$,
where we invoke \thmref{thm:FindRemak-Q}.
For that suppose we have $\omega\theta\in \End G-\Aut G$.  By Fitting lemma we have that:
\begin{equation}
	G=\ker \omega^{\ell(G)} \times \im \omega^{\ell(G)}.
\end{equation}
To compute such a decomposition we compute $\im\omega^{\ell(G)}$ and then apply
{\sc Direct-$\Ops$-Complement} to compute $\ker \omega^{\ell(G)}$.  As $\Ops$ is part of the input, we may test each $\omega\in\Omega$ to find those $\omega$ where $\omega\theta\notin\Aut G$, and with each produce a direct $\Ops$-decomposition.  The restriction of $\omega$ to the constituents induces either zero map, or an automorphism.  Thus the remaining cases are handled by \thmref{thm:FindRemak-Q}. 
\hfill $\Box$

\section{An example}\label{sec:ex}
Here we give an example of the execution of the algorithm for \thmref{thm:FindRemak-Q} which covers several
of the interesting components (but of course fails to address all situations).
We will operate without a specific representation in mind, since we are interested in demonstrating
the high-level techniques of the algorithm for \thmref{thm:FindRemak-Q}.

We trace through how the algorithm might process the group
$$G=D_8 \times Q_8\times \SL(2,5)\times \big(\SL(2,5)\circ \SL(2,5)\big).$$

First the algorithm recurses until it reaches the group 
$$\hat{G}=G/\zeta_2(G)\cong \PSL(2,5)^3.$$
At this point it finds a minimal normal subgroup $N$ of $\hat{G}$, of which there
are three, so we pick $N=\PSL(2,5)\times 1\times 1$.  Next the algorithm computes
a Remak decomposition of $C_G(N)=1\times \PSL(2,5)\times \PSL(2,5)$.  At this point
the algorithm returns the unique Remak decomposition 
	$$\mathcal{Q}:=\{\PSL(2,5)\times 1\times 1,1\times \PSL(2,5)\times 1,1\times 1\times \PSL(2,5)\}.$$
These are pulled back to the set $\{H_1, H_2, H_3\}$ of subgroups in $G$.

Next the algorithm constructs a Remak $G$-decomposition of $\zeta_2(G)$.
For that the algorithm constructs the bilinear map of commutation from 
$\zeta_2(G)/\zeta_1(G)\cong \mathbb{Z}_2^4$
into $\gamma_2(\zeta_2(G))=\langle z_1,z_2\rangle\cong \mathbb{Z}_2^2$, i.e.
$$b:=\Bi(\zeta_2(G)):\mathbb{Z}_2^4\times \mathbb{Z}_2^4\to \mathbb{Z}_2^2$$ 
Below we have described the structure constants for $b$ in a nice basis but remark that
unless we already know the direct factors of $\zeta_2(G)$ it is unlikely to have such a natural form.
\begin{equation}
	b(u,v) = u
	\begin{bmatrix}
		0 & z_1 &  & \\
		-z_1 & 0 & & \\
		& & 0 & z_2 \\
		& & -z_2 & 0
	\end{bmatrix} v^t,\qquad \forall u,v\in \mathbb{Z}_2^4.
\end{equation}
A basis for the centroid of $b$ is computed:
\begin{equation}
	C(b) = \left\{\left(\begin{bmatrix}
		a & 0 &  & \\
		0 & a & & \\
		& & b & 0 \\
		& & 0 & b
		\end{bmatrix},
		\begin{bmatrix} a & 0 \\ 0 & b\end{bmatrix}\right) : 
			a,b\in\mathbb{Z}_2\right\}\cong \mathbb{Z}_2\oplus\mathbb{Z}_2.
\end{equation}
Next, the unique frame 
$\mathcal{E}=\{ (I_2\oplus 0_2, 1\oplus 0), (0_2\oplus I_2,0\oplus 1)\}$ 
of $C(b)$ is built and used to create the subgroups
$\mathcal{K}:=\{D_8\times Z(Q_8), Z(D_8)\times Q_8\}$ in $\zeta_2(G)$.  Here, using
an arbitrary basis $\mathcal{X}$ for $\zeta_1(G)=\mathbb{Z}_2^2\times \mathbb{Z}_4^2$, 
the algorithm {\tt Merge}$(\mathcal{X},\mathcal{K})$
constructs a Remak decomposition $\mathcal{A}:=\{H,K,C_1, C_2\}$ of $\zeta_2(G)$
where $H\cong D_8$, $K\cong Q_8$, and $C_1\cong C_2\cong \mathbb{Z}_4$. 

Finally, the algorithm {\tt Merge}$(\mathcal{A},\mathcal{H})$ returns a 
Remak decomposition of $G$.  To explain the merging process we trace that algorithm
through as well.

Let $R=\SL(d,q)\times 1\times 1$ and $S=1\times (\SL(d,q)\circ\SL(d,q))$.
These groups are directly indecomposable direct factors of $G$ and 
serve as the hypothesized directions of for the direct chain used by {\tt Merge}.
Without loss of generality we index the $H$'s so that $H_2=R\zeta_2(G)$ and $H_1H_3=S\zeta_2(G)$
and 
$$G/\zeta_2(G)=\PSL(d,q)\times \PSL(d,q)\times \PSL(d,q)
	=H_2/\zeta_2(G)\times H_1/\zeta_2(G)\times H_3/\zeta_2(G).$$
Furthermore, $\zeta_2(H_i)=\zeta_2(G)$ for all $i\in\{1,2,3\}$.  Therefore, 
$(\mathcal{A},\mathcal{H})$ satisfies the hypothesis of \thmref{thm:merge}.

The loop in {\tt Merge} begins with $\mathcal{K}_0=\mathcal{A}$ and seeks to extend
$\mathcal{A}$ to $H_1$ by selecting an appropriate subset $\mathcal{A}_1\subseteq \mathcal{K}_0=\mathcal{A}$
and finding a complement $\lfloor H_1\rfloor\leq H_1$ such that 
$\mathcal{K}_1=\mathcal{A}_1\sqcup\{\lfloor H_1\rfloor\}$ is a direct decomposition of $H_1$.
The configuration at this stage is seen in \figref{fig:merge-1}.  By \thmref{thm:Extend}, we have $H,K\in \mathcal{A}_1$ (as those lie outside the center) and one of the $C_i$'s (though no unique choice exists there).

In the second loop iteration we extend $\mathcal{K}_1$ to a $\mathfrak{N}_2$-refined
direct decomposition if $H_1 H_2$.  This selects a subset $\mathcal{A}_2\subseteq \mathcal{K}_1\cap \zeta_2(G)$.  Also $H_1$ and $H_2$ are in different
directions, specifically $H_2=R\zeta_2(G)$ and $H_1\leq S\zeta_2(G)$, so
the algorithm is forced to include $\lfloor H_1\rfloor\in\mathcal{K}_2$ (cf. \thmref{thm:Extend}(iii)) and then creates
a complement $\lfloor H_2\rfloor\cong\SL(2,5)$ to $\langle \mathcal{A}_2,\lfloor H_1\rfloor\rangle$.  
The configuration is illustrated in \figref{fig:merge-2}.  As before, we have $H,K\in\mathcal{K}_2$ as well, but the cyclic groups are now gone as the centers of $\lfloor H_i\rfloor$, $i\in\{1,2\}$, fill out a direct decomposition of $\zeta_2(G)$.

Finally, in the third loop iteration, the direction is back towards $S$ and so 
the extension $\mathcal{K}_3$ of $\mathcal{K}_2$ to $H_1 H_2 H_3$ contains
$\lfloor H_2\rfloor$ and is $\mathfrak{N}_2$-refined.  However, the group $\lfloor H_1\rfloor$
is not a direct factor of $G$ as it is one term in nontrivial central product.  Therefore
that group is replaced by a subgroup $\lfloor H_1 H_3\rfloor\cong \SL(d,q)\circ\SL(d,q)$.  The final configuration
is illustrated in \figref{fig:merge-3}.  $\mathcal{K}_3$ is a Remak decomposition of $G$.

\begin{figure}
\begin{equation*}
\begin{xy}<10mm,0mm>:<0mm,7mm>::
( 0, 0) *+{1}							= "1";
( 2, 1) *+{\zeta_2(\lfloor H_1\rfloor)}	= "zH1f";
( 4, 2) *+{\lfloor H_1\rfloor}			= "H1f";
(-1, 2) *+{\langle \mathcal{A}_1\rangle} 	= "A1";
( 1, 3) *+{\zeta_2(G)}					= "zG";
( 3, 4) *+{H_1}							= "H1";
"1"; 		"zH1f" 		**@{-};
"1"; 		"A1" 		**@{-};
"zH1f"; 		"H1f" 		**@{-};
"zH1f"; 		"zG" 		**@{-};
"A1";		"zG"			**@{-};
"zG";		"H1"			**@{-};
"H1f";		"H1"			**@{-};
\end{xy}
\end{equation*}
\caption{The lattice encountered during the first iteration of the loop in the algorithm 
${\tt Merge}(\mathcal{A},\{H_1,H_2,H_3\})$.}\label{fig:merge-1}
\end{figure}
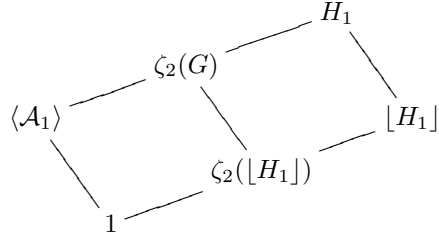

\begin{figure}
\begin{equation*}
\begin{xy}<10mm,0mm>:<0mm,7mm>::
( 0, 0) *+{1}							= "1";
(-2, 1) *+{\zeta_2(\lfloor H_2\rfloor)}	= "zH2f";
(-4, 2) *+{\lfloor H_2\rfloor}			= "H2f";
( 2, 1) *+{\zeta_2(\lfloor H_1\rfloor)}	= "zH1f";
( 0, 2)									= "zH1fzH2f";
(-2, 3) 									= "zH1fH2f";
( 4, 2) *+{\lfloor H_1\rfloor}			= "H1f";
( 2, 3) 									= "H1fzH2f";
( 0, 4) 									= "H1fH2f";
( 0, 1) *+{\langle \mathcal{A}_2\rangle} 	= "A";
(-2, 2) 									= "zH2fA";
(-4, 3) 									= "H2fA";
( 2, 2) 									= "zH1fA";
( 0, 3) *+{\zeta_2(G)}					= "zG";
(-2, 4) *+{H_2}							= "H2";
( 4, 3) 									= "H1fA";
( 2, 4) *+{H_1}							= "H1";
( 0, 5) *+{H_1 H_2}						= "H1H2";
"1"; 		"zH2f" 		**@{-};		"zH2f"; 		"H2f" 		**@{-};
"zH1f";		"zH1fzH2f"	**@{..};		"zH1fzH2f";	"zH1fH2f"	**@{..};
"H1f";		"H1fzH2f"	**@{..};		"H1fzH2f";	"H1fH2f"		**@{..};
"1"; 		"zH1f" 		**@{-}; 		"zH1f";		"H1f"		**@{-};
"zH2f";		"zH1fzH2f"	**@{..}; 	"zH1fzH2f";	"H1fzH2f"	**@{..};
"H2f";		"zH1fH2f"	**@{..}; 	"zH1fH2f";	"H1fH2f"		**@{..};
"A"; 		"zH2fA" 		**@{-}; 		"zH2fA"; 	"H2fA" 		**@{-};
"zH1fA";		"zG"			**@{-};		"zG";		"H2"			**@{-};
"H1fA";		"H1"			**@{-};		"H1";		"H1H2"		**@{-};
"A"; 		"zH1fA" 		**@{-};		"zH1fA";		"H1fA"		**@{-};
"zH2fA";		"zG"			**@{-};		"zG";		"H1"			**@{-};
"H2fA";		"H2"			**@{-};		"H2";		"H1H2"		**@{-};
"1"; 	"A"	**@{-};	"zH2f";	"zH2fA"	**@{-}; "H2f"; 	"H2fA"	**@{-};
"zH1f"; 	"zH1fA"	**@{-};	"zH1fzH2f";	"zG"	**@{..}; "zH1fH2f"; 	"H2"	**@{..};
"H1f"; 	"H1fA"	**@{-};	"H1fzH2f";	"H1"	**@{..}; "H1fH2f"; 	"H1H2"	**@{..};
\end{xy}
\end{equation*}
\caption{The lattice encountered during the second iteration of the loop in the algorithm 
${\tt Merge}(\mathcal{A},\{H_1,H_2,H_3\})$.}\label{fig:merge-2}
\end{figure}
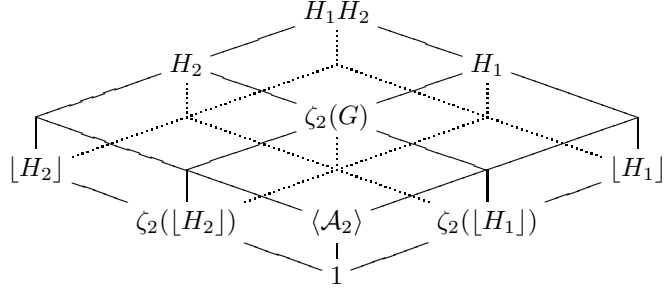

\begin{figure}
\begin{equation*}
\begin{xy}<10mm,0mm>:<0mm,7mm>::
( 0, 0) *+{1}							= "1";
(-2, 1) *+{\zeta_2(\lfloor H_2\rfloor)}	= "zH2f";
(-4, 2) *+{\lfloor H_2\rfloor}			= "H2f";
( 2, 1) *+{\zeta_2(\lfloor H_1 H_3\rfloor)}= "zH13f";
( 0, 2) 									= "zH13fzH2f";
(-2, 3) 									= "H2fzH13f";
( 4, 2) 									= "H1f";
( 2, 3) 									= "H1fzH2f";
( 0, 4) 									= "H1fH2f";
(-0.5, 1) *+{\langle \mathcal{A}_3\rangle} 	= "A";
(-2.5, 2) 									= "zH2fA";
(-4.5, 3) 									= "H2fA";
(1.5, 2) 									= "zH13fA";
(-0.5, 3) *+{\zeta_2(G)}						= "zG";
(-2.5, 4) *+{H_2}							= "H2";
(3.5, 3) 									= "H1fA";
(1.5, 4) *+{H_1}								= "H1";
(-0.5, 5) 									= "H1H2";
%
( 2, 2) 										= "H3f";
(1.5, 3) 									= "H3fA";
(-0.5, 4) *+{H3}								= "H3";
(-2.5, 5) *+{H_2H_3}							= "H2H3";
( 4, 3)	*+{\lfloor H_1 H_3\rfloor}			= "H13f";
(3.5, 4) 									= "H13fA";
(1.5, 5) *+{H_1H_3}							= "H1H3";
(-0.5, 6) *+{G}								= "G";
"1"; 		"zH2f"		**@{-};	"zH2f";		"H2f"		**@{-};
"zH13f";		"zH13fzH2f"	**@{..};"zH13fzH2f";	"H2fzH13f"	**@{..};
"H1f"; 		"H1fzH2f"	**@{..};"H1fzH2f";	"H1fH2f"		**@{..};
"1"; 		"zH13f"		**@{-};"zH13f";		"H1f"		**@{-};
"zH2f";		"zH13fzH2f"	**@{..};"zH13fzH2f";	"H1fzH2f"	**@{..};
"H2f";		"H2fzH13f"	**@{..};"H2fzH13f";	"H1fH2f"		**@{..};
"A"; 		"zH2fA"		**@{-};"zH2fA";		"H2fA"		**@{-};
"zH13fA";	"zG"			**@{-};"zG";		"H2"			**@{-};
"H1fA"; 		"H1"			**@{..};"H1";		"H1H2"		**@{..};
"A"; 		"zH13fA"		**@{-};"zH13fA";	"H1fA"		**@{..};
"zH2fA";		"zG"			**@{-};"zG";		"H1"			**@{..};
"H2fA";		"H2"			**@{-};"H2";		"H1H2"		**@{..};
"H3f";		"H3fA"		**@{-};"H3fA";		"H3"			**@{-};
"H3"; 		"H2H3"		**@{-};"H13f";		"H13fA"		**@{-};
"H13fA";		"H1H3"		**@{-};"H1H3";		"G"			**@{-};
"H3f";		"H13f"		**@{-};"H3fA";		"H13fA"		**@{-};
"H3";		"H1H3"		**@{-};"H2H3";		"G"			**@{-};
"1";			"A"			**@{-};"zH2f";		"zH2fA"		**@{-};
"H2f";		"H2fA"		**@{-};"zH13f";		"zH13fA"		**@{-};
"zH13fA";	"H3fA"		**@{-};"H2fzH13f";	"H2"			**@{..};
"H2";		"H2H3"		**@{-};"zH13fzH2f";	"zG"			**@{..};
"zG";		"H3"			**@{-};"H1fH2f";	"H1H2"		**@{..};
"H1H2";		"G"			**@{..};"H1fzH2f";	"H1"			**@{..};
"H1"; 		"H1H3"		**@{..};"H1f";		"H1fA"		**@{..};
"H1fA";		"H13fA"		**@{..};"zH13f";		"H3f"		**@{-};
"H1f";		"H13f"		**@{-};
\end{xy}
\end{equation*}
\caption{The lattice of encountered during the third iteration of the loop in the algorithm 
${\tt Merge}(\mathcal{A},\{H_1,H_2,H_3\})$.}\label{fig:merge-3}
\end{figure}
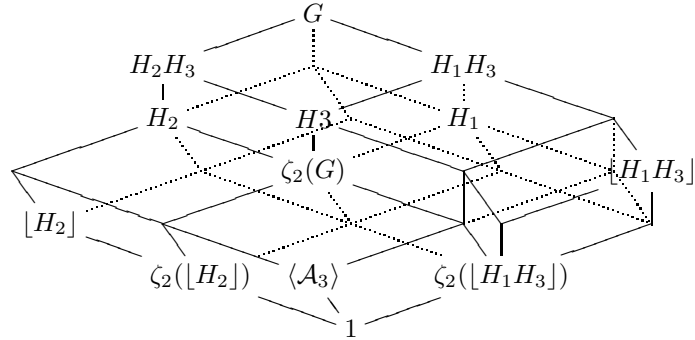

\section{Closing remarks}\label{sec:closing}

Historically the problem of finding a Remak decomposition focused on groups given by their multiplication table since even there there did not seem to be a polynomial-time solution.  It was known that a Remak decomposition could be found by checking all partitions of all minimal generating sets of a group $G$ and so the problem had a sub-exponential complexity of $|G|^{\log |G|+O(1)}$.  That placed it in the company of other interesting problems including testing for an isomorphism between two groups \cite{Miller:nlogn}.   Producing an algorithm that is polynomial-time in the size of the group's multiplication table (i.e. polynomial in $|G|^2$) was progress, achieved independently in \cite{KN:direct} and \cite{Wilson:thesis}.  Evidently, \thmref{thm:FindRemak} provides a polynomial-time solution for groups input in this way (e.g. use a regular representation). With a few observations we sharpen \thmref{thm:FindRemak} in that specific context to the following:
\begin{thm}\label{thm:nearly-linear}
There is a deterministic nearly-linear-time algorithm which, given a group's multiplication table,
returns a Remak decomposition of the group.
\end{thm}
\begin{proof}
The algorithm for \thmref{thm:FindRemak-Q} is polynomial in $\log |G|$.  As the input length here is $|G|^2$, it suffices to show that the problems listed in Section \ref{sec:tools} have $O(|G|^2\log^c |G|)$-time or better solutions.  Evidently, \textalgo{Order}, \textalgo{Member}, \textalgo{Solve} each have brute-force linear-times solutions.  \textalgo{Presentation} can be solved in linear-time by selecting a minimal generating set $\{g_1,\dots,g_{\ell}\}$ (which has size $\log |G|$) and acting on the cosets of $\{\langle g_i,\dots, g_{\ell}\rangle : 1\leq i\leq \ell\}$ produce defining relations of the generators in fashion similar to \cite[Exercise 5.2]{Seress:book}.  For \textalgo{Minimal-Normal}, begin with an element and takes it normal closure.  If this is a proper subgroup recurse, otherwise, try an element which is not conjugate to the first and repeat until either a proper normal subgroup is discovered or it is proved that group is simple. That takes $O(|G|^2)$-time.  The remaining algorithms \textalgo{Primary-Decomposition}, \textalgo{Irreducible}, and \textalgo{Frame} have brute force linear-time solutions.  Thus, the algorithm can be modified to run in times $O(|G|^2 \log^c |G|)$.
\end{proof}

Section \ref{sec:lift-ext} lays out a framework which permits for a local view of the direct products of group.  We have some lingering questions in this area.  
\begin{enumerate}
\item What is the best series of subgroups to use for the algorithm of \thmref{thm:FindRemak-Q}?  

Corollaries \ref{coro:canonical-graders} and \ref{coro:canonical-grader-II} offer alternatives series to use in the algorithm.  There is an option for a top-down algorithm based on down graders.  That may allow for a black-box algorithm since verbal subgroups can be constructed in black-box groups; see \cite[Section 2.3.4]{Seress:book}.

\item Is their a parallel NC solution for \textalgo{Remak-$\Ops$-Decomposition}?  

We can speculate how this may proceed.  First, select an appropriate series $1\leq G_1\leq\cdots \leq G_n=G$ for $G$ and distribute and use parallel linear algebra methods to find Remak decompositions $\mathcal{A}_{i0}$ of each $G_{i+1}/G_{i}$, for $1\leq i<n$.  Then for $0\leq j\leq \log n$, for each $1\leq i\leq n/2^j$ in parallel compute $\mathcal{A}_{i(j+1)}:=\textalgo{Merge}(\mathcal{A}_{ij},\mathcal{A}_{(i+1)j})$.  When $j=\lfloor \log n\rfloor$ we have a direct decomposition $\mathcal{A}_{1\log n}$ of $G$ and have used poly-logarithmic time. Unfortunately, \thmref{thm:merge}(a) is not satisfied in these recursions, so we cannot be certain that the result is a Remak decomposition.
\end{enumerate}

\section*{Acknowledgments}
I am indebted to W. M. Kantor for taking a great interest in this work and offering guidance.  Thanks to E. M. Luks, C.R.B. Wright, and \'{A}. Seress for encouragement and many helpful remarks.


\def\cprime{$'$}

\end{document}